\documentclass[10pt]{amsart}
\usepackage{amsmath,amssymb,verbatim}
\usepackage{titlesec}
\titleformat{\section}[hang]
{\upshape\bfseries}{\thesection.}{5pt}
{\large\bfseries}
\usepackage[utf8]{inputenc}
\usepackage{enumerate}
\usepackage{pdflscape}
\usepackage{amsthm}
\usepackage{mathrsfs}
\usepackage{color}
\usepackage[normalem]{ulem}
\usepackage{cancel}
\usepackage{tikz-cd}
\usepackage{tikz}
\usetikzlibrary{matrix}
\usepackage[all]{xy}
\usepackage{mathtools}
\usepackage[bookmarks=true,colorlinks=true,linkcolor=black,citecolor=black,filecolor=black,urlcolor=black]{hyperref}
\usepackage{cleveref}
\usepackage{arydshln}
\usepackage{enumitem}
\usepackage{amsaddr}

\topskip=\baselineskip 
\newtheorem{theorem}{Theorem}[section]
\newtheorem{lemma}[theorem]{Lemma}
\newtheorem{proposition}[theorem]{Proposition}

\newtheorem{definition}[theorem]{Definition}
\newtheorem{corollary}[theorem]{Corollary}
\newtheorem{remark}[theorem]{Remark}

\newtheorem{algorithm}{Algorithm}

\newtheorem{maintheorem}{Theorem}

\makeatletter
\@addtoreset{equation}{section}
\makeatother

\newcommand{\Aut}{\mbox{\rm Aut}}
\newcommand{\Inn}{\mbox{\rm Inn}}
\newcommand{\T}{\mbox{\rm T}}

\newcommand{\inv}[1]{[#1]}
\newcommand{\INV}{\mbox{\rm MCINV}}
\newcommand{\IN}{\mbox{\rm IN}}
\newcommand{\N}{{\mathbb N}}
\newcommand{\Z}{{\mathbb Z}}

\newcommand{\G}{\mathcal{G}}

\newcommand{\GEN}[1]{\left\langle #1 \right\rangle}

\newcommand{\U}{\mathcal{U}}

\newcommand{\Ese}[2]{\mathcal{S}\left(#1\mid #2\right)}

\newenvironment{proofof}{\par\noindent \textit{Proof of }}{\qed\par\bigskip}
\newenvironment{proofsof}{\par\noindent \textit{Proofs of }}{\qed\par\bigskip}

\DeclareMathOperator{\lcm}{lcm}

\newcommand{\qand}{\quad \text{and} \quad}

\DeclareMathOperator{\Res}{Res}

\title{A classification of metacyclic groups by group invariants}

\author{Àngel Garc\'{\i}a-Blázquez and \'{A}ngel del R\'{\i}o}
\thanks{Partially supported Grant PID2020-113206GB-I00 funded by MCIN/AEI/10.13039/501100011033.}

\address{Departamento de Matem\'{a}ticas, Universidad de Murcia, 30100, Murcia, Spain} \email{angel.garcia11@um.es, adelrio@um.es}

\keywords{Finite groups, Metacyclic groups}

\subjclass{20D10,20E99}

\begin{document}
	
	\maketitle

	{\centering Dedicated to Toma Albu and Constantin N\u{a}st\u{a}sescu in their 80th birthday\par}

	\begin{abstract}
	We obtain a new classification of the finite metacyclic group in terms of group invariants. We present an algorithm to compute these invariants, and hence to decide if two given finite metacyclic groups are isomorphic, and another algorithm which computes all the metacyclic groups of a given order. A GAP implementation of these algorithms is given.
	\end{abstract}

	\section{Introduction}
	
	Classifying groups is a fundamental problem in group theory.
	Unfortunately it is a task which seems out of reach except for restricted families of groups.
	One of the classes which have received much attention is that of finite metacyclic groups. 
	It is well known that every finite metacyclic group has a presentation of the following form 
	$$\G_{m,n,s,t}=\GEN{a,b \mid a^m=1, b^n=a^s, a^b=a^t}$$
	for natural numbers $m,n,s,t$ satisfying $s(t-1)\equiv t^n-1 \equiv 0 \mod m$.
	However, the parameters $m,n,s$ and $t$ are not invariants of the group. 
	Traditionally the authors dealing with the classification of finite metacyclic group select distinguished values of $m,n,s$ and $t$ so that each isomorphism class is described by a unique election of the parameters (see \cite{Zassenhaus1958,Hall1959,Beyl1972,King1973,Liedahl1994,Liedahl1996,NewmanXu1988,Redei1989,Lindenberg1971,Sim1994}).
	This approach was culminated by C.E. Hempel who presented a classification of all the finite metacyclic groups in \cite{Hempel2000}. However it is not clear how to use this classification to describe the distinguished parameters identifying a given metacyclic group and how those distinguished parameters are connected with group invariants. 
	
	The aim of this paper is to present an alternative classification of the finite metacyclic using a slightly different approach in terms of group invariants which allows an easy implementation.
	Namely, we associate to every finite metacyclic group $G$ a $4$-tuple
	$\INV(G)=(m_G,n_G,s_G,\Delta_G)$ where $m_G,n_G$ and $s_G$ play the role of $m,n$ and $s$ in the presentation above and $\Delta_G$ is a cyclic subgroup of units modulo a divisor of $m_G$.
	Our main result consists in proving that $\INV(G)$ is an invariant of the group $G$ which determines $G$ up to isomorphism, i.e. if $G$ and $H$ are two finite metacyclic groups then they are isomorphic if and only if $\INV(G)=\INV(H)$ (\Cref{Invariants}). 
	Moreover, we describe in \Cref{Parameters} the possible values $(m,n,s,\Delta)$ of $\INV(G)$ and for such value we show how to find an integer $t$ such that $\INV(\G_{m,n,s,t})=(m,n,s,\Delta)$ (\Cref{Construction}). 
	This allows a computer implementation of the following function: one which computes $\INV(G)$ for any given finite metacyclic group, and hence of another function which decide whether two metacyclic groups are isomorphic, and another one which computes all the metacyclic subgroups of a given order.

To define $\INV(G)$ we need to introduce some notation. 
First of all, we adopt the convention that $0$ is not a natural number, so $\N$ denotes the set of positive integers. Moreover by a prime we mean a prime in $\N$. 
If  $m \in \N$, $p$ is a prime, $\pi$ is a set of primes and $A$ a finite abelian group then we denote
	\begin{center}
	\begin{tabular}{rcl}
	$\pi(m)$ & $=$ & set of primes dividing $m$, \\
	$\U_m$ &$=$& group of units of the ring $\Z/m\Z$, \\
	$m_p$ & $=$ & maximum power of $p$ dividing $m$,\\
	$m_{\pi}$ & $=$ & $\prod_{p\in \pi} m_p$, \\
	$A_{\pi}$ &=& Hall $\pi$-subgroup of $A$, \\
	$A_{\pi'}$ &=& Hall $\pi'$-subgroup of $A$.
	\end{tabular}
	\end{center}		
If $t\in \Z$ with $\gcd(t,m)=1$ then $[t]_m$ denotes the element of $\U_m$ represented by $t$ and $\GEN{t}_m$ denotes the subgroup of $\U_m$ generated by $[t]_m$.
If $q\mid m$ then $\Res_q:\U_m\rightarrow \U_q$ denotes the natural map, i.e. $\Res_q([t]_m)=[t]_q$.

Let $T$ be a cyclic subgroup of $\U_m$.
Then we define $\inv{T}=(r,\epsilon,o)$
\begin{eqnarray*}
r&=& \text{greatest divisor of } m \text{ such that } \Res_{r_{2'}}(T)=1 \text{ and }\Res_{r_2}(T)\subseteq \GEN{-1}_{r_2}; \\
\epsilon&=&
\begin{cases}
-1, & \text{if } \Res_{r_2}(T)\ne 1; \\
1, & \text{otherwise}.
\end{cases} \\
o&=&|\Res_{m_{\nu}}(T_{\nu'})|, \text{ with } \nu = \pi(m)\setminus \pi(r).
\end{eqnarray*}
If moreover, $n,s\in \N$ then we denote
	$$[T,n,s]=m_{\nu} \prod_{p\in \pi(r)} m'_p$$
with $m'_p$ defined as follows:
\begin{equation}\label{m'}
\begin{split}
& \text{if }\epsilon^{p-1}=1 \text{ then }
m'_p = \min\left(m_p,o_pr_p,\max\left(r_p,s_p,r_p\frac{s_po_p}{n_p}\right)\right);\\
& \text{if }\epsilon=-1 \text{ then }
m'_2=\begin{cases} r_2, & \text{if either } o_2\le 2  \text{ or } m_2\le 2r_2; \\
\frac{m_2}{2}, & \text{if } 4\le o_2<n_2, 4r_2\le m,  \text{ and if } s_2\le n_2r_2 \text{ then } s_2=m_2<n_2r_2; \\
m_2, & \text{otherwise}.
\end{cases}
\end{split}
\end{equation}

Let $A$ be a cyclic group of order $m$.
Then the map $\sigma_A:\U_m \rightarrow \Aut(A)$ associating $[r]_m$ with the map $a\mapsto a^r$, is a group isomorphism.
If moreover $A$ is a normal subgroup of a group $G$ then we define
	$$T_G(A)=\sigma_A^{-1}(\Inn_G(A)),$$
where $\Inn_G(A)$ is formed by the restriction to $A$ of the inner automorphisms of $G$.
We introduce notation for the entries of $T_G(A)$ by setting
$$(r_G(A),\epsilon_G(A),o_G(A))=\inv{T_G(A)}.$$
\begin{definition}
Let $G$ be a group. A \emph{metacyclic kernel} of $G$ is a normal subgroup $A$ of $G$ such that $A$ and $G/A$ are cyclic.
A \emph{metacyclic factorization} of a group $G$ is an expression $G=AB$ where $A$ is a normal cyclic subgroup of $G$ and $B$ is a cyclic subgroup of $G$. 

A \emph{minimal kernel} of $G$ is a kernel of $G$ of minimal order.

A metacyclic factorization $G=AB$ is said to be \emph{minimal} in $G$ if $(|A|,r_G(A),[G:B])$ is minimal in the lexicographical order.
In that case we denote $m_G=|A|$, $n_G=[G:A]$, $s_G=[G:B]$ and $r_G=r_G(A)$.
\end{definition}

Clearly a group is metacyclic if and only if it has metacyclic kernel if and only if it has a metacyclic factorization. Sometimes we abbreviate metacyclic kernel of $G$ or metacyclic factorization of $G$ and we simply say kernel of $G$ or factorization of $G$.

If $G=AB$ is a metacyclic factorization of $G$ then we denote
$$\Delta(AB) = \Res_{[T,n,s]}(T),  \quad \text{ with } \quad T=T_G(A),\quad  n=[G:A] \qand s=[G:B].$$

We will prove that $\Delta(AB)$ is constant for all the minimal metacyclic factorizations (\Cref{DeltaWD}). This allows to define the desired invariant:
	$$\INV(G)=(|A|,[G:A],[G:B],\Delta(AB)), \text{ with } G=AB \text{ minimal factorization of } G.$$

Our first result states that $\INV(G)$ determines $G$ up to isomorphisms, formally:

\begin{maintheorem}\label{Invariants}
Two finite metacyclic groups $G$ and $H$ are isomorphic if and only if $\INV(G)=\INV(H)$.
\end{maintheorem}

Our next result describes the values realized as $\INV(G)$ with $G$ a finite metacyclic group.

\begin{maintheorem}\label{Parameters}
Let $m,n,s\in \N$ and let $\Delta$ be a cyclic subgroup of $\U_{m'}$ with $m'\mid m$.
Let $\inv{\Delta}=[r,\epsilon,o]$ and $\nu=\pi(m)\setminus \pi(r)$. Then the following conditions are equivalent:
\begin{enumerate}
 \item $(m,n,s,\Delta)=\INV(G)$ for some finite metacyclic group $G$.
 \item \begin{enumerate}
        \item\label{ParamB} $s$ divides $m$, $|\Delta|$ divides $n$ and $m_{\nu}=s_{\nu}=m'_{\nu}$.
        \item\label{Param'} \eqref{m'} holds for every $p\in \pi(r)$.
        \item\label{Param-} If $\epsilon=-1$ then $\frac{m_2}{r_2}\le n_2$, $m_2\le 2s_2$ and $s_2\ne n_2r_2$. If moreover $4\mid n$, $8\mid m$ and $o_2<n_2$ then $r_2\le s_2$.
			\item\label{Param+} For every $p\in \pi(r)$ with $\epsilon^{p-1}=1$, we have
			$\frac{m_p}{r_p}\le s_p\le n_p$ and if $r_p> s_p$ then  $n_p<s_po_p$;
	\end{enumerate}
\end{enumerate}
\end{maintheorem}

Our last result  shows how to construct a metacyclic group $G$ with given $\INV(G)$:
If $m,n,s\in \N$ with $s\mid m$ then we define the following subgroup of $\U_m$:
$$\U_m^{n,s} = \{ [t]_m : m \mid s(t-1), \qand t^n \equiv 1 \mod m\}.$$
If $T$ is a cyclic subgroup of $\U_m^{n,s}$ generated by $[t]_m$ then we denote
	$$\G_{m,n,s,T} = \G_{m,n,s,t} = \{a,b: a^m=1, b^n=a^s, a^b=a^t\}.$$
It is easy to see that the isomorphism type of this group is independent of the election of the generator $[t]_m$  of $T$ (\Cref{ParametersTDelta}.\eqref{mDividernrs}). Moreover, the assumption $T\subseteq \U_m^{n,s}$ warranties that $|a|=m$, $|\G_{m,n,s,T}|=mn$ and $|b|=\frac{mn}{s}$.

\begin{remark}\label{ConsRemark}
Suppose that $m,n,s$ and $\Delta\le \U_{m'}$ satisfy the conditions of statement (2) in \Cref{Parameters} and $\inv{\Delta}=(r,\epsilon,o)$.
Then $\Res_{m'_p}(\Delta)=\GEN{\epsilon^{p-1}+r_p}_{m'_p}$ for every $p\in\pi(r)$ and hence there is an integer $t'$ such that $\Delta=\GEN{t'}_{m'}$ and $t'\equiv \epsilon^{p-1}+r_p\mod m'_p$ for every $p\in \pi(r)$. Using the Chinese Remainder Theorem we can select an integer $t$ such that $t\equiv t'\mod m'$ and $t\equiv \epsilon^{p-1}+r_p \mod m_p$ for every $p\in \pi(r)$ and let $T=\GEN{t}_m$. 
Then $T \subseteq \U_n^{n,s}$, $\Res_{m'}(T)=\Delta$ and $[T]=[\Delta]$.
Then the following theorem ensures that $\INV(G_{m,n,s,T})=(m,n,s,\Delta)$.
\end{remark}

\begin{maintheorem}\label{Construction}
Let $m,n,s\in \N$ and let $\Delta$ be a cyclic subgroup of $\U_{m'}$ with $m'\mid m$. Suppose that they satisfy the conditions of (2) in \Cref{Parameters} and let $T$ be a cyclic subgroup of $\U_m^{n,s}$ such that $[T]=[\Delta]$ and $\Res_{m'}(T)=\Delta$.
Then $(m,n,s,\Delta)=\INV(\G_{m,n,s,T})$.
\end{maintheorem}

For implementation it is convenient to replace the fourth entry of $\INV(G)$ by a distinguished integer $t_G$ so that $G\cong \G_{m_G,n_G,s_G,t_G}$  and $G\cong H$ if and only if $(m_G,n_G,s_G,t_G)=(m_H,n_H,s_H,t_H)$. 
We select $t_G$ satisfying the conditions of \Cref{ConsRemark}. In particular,  $[t_G]_{m_{\pi}}$  is uniquely determined by the condition $t\equiv \epsilon^{p-1}+r_p\mod m_p$ for every $p\in \pi(r)$. However there is not any natural election of $[t_G]_{m_{\pi'}}$ and we simply take the minimum possible value. More precisely, if $(m,n,s,\Delta)=\INV(G)$, $(r,\epsilon,o)=\inv{\Delta}$ and $m'$ is given by \eqref{m'} then define
$$t_G = \min \{t\ge 0 : \Res_{m'}(\GEN{t}_m)=\Delta \qand
t \equiv \epsilon^{p-1}+r_p \mod m_p \text{ for every } p\in \pi(r)\}.$$
We call $(m_G,n_G,s_G,t_G)$ the list of \emph{metacyclic invariants} of $G$.
Clearly if $H$ is another metacyclic group then $G\cong H$ if and only if $G$ and $H$ have the same metacyclic invariants.
Moreover, by \Cref{Construction}, if $(m,n,s,t)$ is the list of metacyclic invariants of $G$ then $G\cong \G_{m,n,s,t}$. 

We outline the contains of the paper: 
In \Cref{SectionNotation} we introduce the general notation, not mentioned in this introduction, and present some preliminary technical results. In \Cref{SectionMF} we prove several lemmas on metacyclic factorizations aiming to an intrinsic description of when a metacyclic factorization is minimal. It includes an algorithm to obtain a minimal metacyclic factorization from an arbitrary one. This section concludes with \Cref{MetaTh} which is the keystone to prove \Cref{Invariants}, \Cref{Parameters} and \Cref{Construction} in \Cref{SectionProofs}. In  \Cref{SectionImplementation} we introduce an algorithm to compute the metacyclic invariants of a given metacyclic group and use this to decide if two metacyclic groups are isomorphic, and another algorithm to construct all the metacyclic groups of a given order. We present also implementations in GAP \cite{GAP} of these algorithms.

\section{Notation and preliminaries}\label{SectionNotation}

By default all the groups in this paper are finite. 
We use standard notation for a group $G$:
$Z(G)=$ center of $G$, $G'=$ commutator subgroup of $G$, $\Aut(G)=$ group of automorphisms of $G$.
If $g,h\in G$ then 
$|g|=$ order of $g$, $g^h=h^{-1}gh$, $[g,h]=g^{-1}g^h$.
If $\pi$ is a set of primes then $g_{\pi}$ and $g_{\pi'}$ denote the $\pi$-part and $\pi'$-part of $g$, respectively.
When $p$ is a prime we rather write $g_p$ and $g_{p'}$ than $g_{\{p\}}$ and $g_{\{p\}'}$, respectively.
Similarly, if $G$ is a finite abelian group then $G_p$ and $G_{p'}$ denote the $p$-part of $G$ and the $p'$-part of $G$, respectively.

Let $G$ be a metacyclic group.
Observe that $A$ is a kernel of $G$ if and only if $G$ has a metacyclic factorization of the form $G=AB$. In that case, if
	$$m=|A|, \quad n=[G:A], \quad s=[G:B] \qand T=T_G(A)=\GEN{t}_m,$$
then $s\mid m$, $|B|=n\frac{m}{s}$, $T\subseteq \U^{n,s}_m$ and $A$ and $B$ have generators $a$ and  $b$, respectively, such that $b^n=a^s$ and $a^b=a^t$. Thus $G\cong \G_{m,n,s,T}$.

	If $p$ is a prime then $v_p$ denotes  the $p$-adic valuation on the integers. 
	
	Let $a\in \Z$ and $m\in \N$. If $\gcd(a,m)=1$ then $o_m(a)$ denotes the order of $[a]_m$ i.e. $o_m(a)=\min\{n\in \N : a^n \equiv 1 \mod m\}$. 
	If $a\ne 0$ then we denote
	$$\Ese{a}{m}=\sum_{i=0}^{m-1} a^i = 
	\begin{cases} m, & \text{ if } a=1; \\
	\frac{a^m-1}{a-1}, & \text{otherwise}.
	\end{cases}$$
	This notation occurs in the following statement where $g$ and $h$ are elements of a group: 
	\begin{equation}\label{Potencia}
	\text{If } g^h=g^a \text{ then } (hg)^m = h^m g^{\Ese{a}{m}}.
	\end{equation}

The following lemma collects some useful properties of the operator $\Ese{-}{-}$ which will be used throughout.

\begin{lemma}\label{PropEse}
	Let $p,R,m\in N$ with $p$ prime and suppose that $R\equiv 1 \mod p$.
	\begin{enumerate}
		\item Suppose that either $p\ne 2$ or $p=2$ and $R\equiv 1 \mod 4$.
		Then
		\begin{enumerate}
			\item\label{vpGeneral} $v_p(R^m-1)=v_p(R-1)+v_p(m)$ and $v_p(\Ese{R}{m})=v_p(m)$.
			\item\label{opGeneral} $o_{p^m}(R)=p^{\max(0,m-v_p(R-1))}$.
			\item\label{PotenciasGeneral} If $a=v_p(R-1)\le m$ then
			$\GEN{R}_{p^m}=\{[1+yp^a]_{p^m} : 0\le y < p^{m-a}\}$.
		\end{enumerate}
		\item Suppose that $R\equiv -1 \mod 4$. Then
		\begin{enumerate}
			\item\label{vpEspecial} $v_2(R^m-1)=\begin{cases} v_2(R+1)+v_2(m), & \text{if } 2\mid m;\\ 1, & \text{otherwise};\end{cases}$
			\\ and
			$v_2(\Ese{R}{m})=\begin{cases} v_2(R+1)+v_2(m)-1, & \text{if } 2\mid m;\\ 0, & \text{otherwise};\end{cases}$.
			\item\label{opEspecial} $o_{2^m}(R)=\begin{cases} 1, & \text{if } m\le 1; \\ 2^{\max(1,m-v_2(R+1))}, & \text{otherwise}\end{cases}$.
			\item\label{vpEspecial+} $v_2(R^m+1) = \begin{cases} v_2(R+1), & \text{if } 2\nmid m; \\ 1, & \text{otherwise}. \end{cases}$.
		\end{enumerate}
	\end{enumerate}
\end{lemma}

\begin{proof}
	\eqref{vpGeneral} The first equality can be easily proven by induction on $m$. Then the second follows from $R^m-1=(R-1)\Ese{R}{m}$.

	\eqref{opGeneral} is a direct consequence of \eqref{vpGeneral}.

	\eqref{PotenciasGeneral} By \eqref{vpGeneral} we have $\GEN{R}_{p^m}\subseteq \{[1+yp^a]_{p^m} : 0\le y < p^{m-a}\}$ and by \eqref{opGeneral} the first set has $p^{m-a}$ elements. As the second one has the same cardinality, equality holds.

	\eqref{vpEspecial} Suppose that $R\equiv -1 \mod 4$.
	If $2\nmid m$ then $R^m\equiv -1 \mod 4$ and hence $v_2(R^m -1)=1$.
	As $R^2 \equiv 1 \mod 4$, if $2\mid m$ then, by \eqref{vpGeneral} we have $v_2(R^m-1) = v_2((R^2)^\frac{m}{2}-1) = v_2(R^2-1)+v_2\left(\frac{m}{2}\right) = v_2(R+1)+v_2(m)$.
	This proves the first part of \eqref{vpEspecial}. Then the second part follows from $R^m-1=(R-1)\Ese{R}{m}$.

	\eqref{opEspecial} follows easily from \eqref{vpEspecial}.

	\eqref{vpEspecial+} Since $R$ is odd, both $R^m-1$ and $R^m+1$ and are even and exactly one of  $v_2(R^m-1)$ and $v_2(R^m+1)$ equals $1$.
	Thus, from \eqref{vpEspecial} we deduce that if $2\mid m$ then $v_2(R^m+1)=1$.
	Suppose otherwise that $m$ is odd and greater than $2$.
	Then $v_2(R^{m-1}-1)=v_2(R+1)+v_2(m-1)>v_2(R+1)$, so that $v_2(R^m+1)=v_2(R(R^{m-1}-1+1)+1) = v_2(R+1+R(R^{m-1}-1))=v_2(R+1)$.

\end{proof}

The following lemma follows by straightforward arguments.

\begin{lemma}\label{ParametersTDelta}
	Let $m,n,s\in \N$, let $T$ be a cyclic subgroup of $\U_m$, and denote $(r,\epsilon,o)=\inv{T}$, $m'=[T,n,s]$ and $\Delta=\Res_{m'}(T)$. 
	\begin{enumerate}
		\item\label{DetAreo} If $T=\GEN{t}_m$ then $|T|=o_m(t)$, $r_{2'}=\gcd(m_{2'},t-1)$,  $r_2=\max(\gcd(m_2,t-1),\gcd(m_2,t+1))=\gcd(m_2,t-\epsilon)$ and $o=o_{m_{\nu}}(t)_{\nu'}$ with $\nu=\pi(m)\setminus \pi(r)$.
		\item $r\mid m'\mid m$ and $\pi(m)=\pi(m')$.
		\item $\inv{T}=\inv{\Delta}$.
		\item For every $p\in \pi(r)$ we have $\Res_{m_p}(T_p)=\GEN{\epsilon^{p-1}+r_p}_{m_p}$ and 
		$$|\Res_{m_p}(T_p)| = \begin{cases} 2, & \text{if } p=2, \epsilon=-1 \text{ and } r_2=m_2; \\ \frac{m_p}{r_p}, & \text{otherwise}.\end{cases}$$
		\item\label{mDividernrs} If $s\mid m$ and $T\subseteq \U_m^{n,s}$ then
		$m_{\pi(r)}\mid rn$, $m_{\pi(r)} \mid rs$, $o \mid n_{\pi(m)\setminus \pi(r)}$ and if $\epsilon=-1$  then
		$m_2 \in \{s_2,2s_2\}$.
		If moreover $T=\GEN{t}_m=\GEN{u}_m$ then there is a $k\in \N$ with $\gcd(k,|T|)=1$ and $a\mapsto a^k$, $b\mapsto b^k$ defines an isomorphism $\G_{m,n,s,t}\rightarrow \G_{m,n,s,u}$.
	\end{enumerate}
\end{lemma}

\begin{definition}
Given $m,n,s\in \N$ with $s\mid m$ and a cyclic subgroup of $\U_m$, we say that $T$ is $(n,s)$-\emph{canonical} if $T\subseteq \U_m^{n,s}$ and if $(r,\epsilon,o)=\inv{T}$ then the following conditions are satisfied:
\begin{itemize}
	\item[(Can--)]\label{MinMF-} If $\epsilon=-1$ then $s_2\ne r_2n_2$. If moreover, $m_2\ge 8$, $n_2\ge 4$, $o_2<n_2$ then $r_2\le s_2$.	
	\item[(Can+)]\label{MinMF+} For every $p\in \pi$ with $\epsilon^{p-1}=1$ we have $s_p\mid n$ and $r_p\mid s$ or $s_po_p\nmid n$.
\end{itemize}
\end{definition}

\section{Metacyclic factorizations}\label{SectionMF}

In this section $G$ is a finite metacyclic group.
Moreover we fix the following notation:
\begin{eqnarray*}
\pi &=& \text{set of prime divisors of } |G| \text{ such that } G \text{ has a normal Hall } p'\text{-subgroup}, \\
\pi'&=&\pi(|G|)\setminus \pi, \\
o_G &=& |\Inn_G({G'}_{\pi'})|_{\pi}.
\end{eqnarray*}

In our first lemma we show that $\pi, \pi'$ and $o_G$ are determined by any kernel of $G$.

\begin{lemma}\label{Basic}
Let $G=AB$ be a metacyclic factorization and let $m=|A|$, $s=[G:A]$, $r=r_G(A)$ and $o=o_G(A)$.
Then
\begin{enumerate}
	\item\label{BasicHall} For every set of primes $\mu$, $A_\mu B_\mu$ is a Hall $\mu$-subgroup of $G$.
	\item\label{Basicpi} $p\in \pi'$ if and only if $G' \setminus Z(G)$ has an element of order $p$ if and only if $A \setminus Z(G)$ has an element of order $p$.
	\item\label{BasicG'} ${G'}_{\pi'}=A_{\pi'}$ and $A_{\pi'}\cap B_{\pi'}=1$.
	\item\label{BasicResume} $\pi'=\pi(m)\setminus \pi(r)$, $s_{\pi'}=m_{\pi'}$ and $o=o_G$.
	\item\label{BasicSemi} $G=A_{\pi'}\rtimes \left(B_{\pi'} \times \prod_{p\in \pi} A_pB_p\right)$. In particular $[B_{p'},A_p]=1$ for every $p\in \pi$.
\end{enumerate}
\end{lemma}

\begin{proof}
\eqref{BasicHall} As $A$ is normal in $G$, $A_\mu B_\mu$ is a $\mu$-subgroup of $G$ and $A_{\mu'}B_{\mu'}$ is a $\mu'$-subgroup of $G$. Moreover $G=AB=A_\mu B_\mu A_{\mu'} B_{\mu'}$ and hence $[G:A_\mu B_\mu]=|A_{\mu'}B_{\mu'}|$. Thus $A_\mu B_\mu$ is a Hall $\mu$-subgroup of $G$.

\eqref{Basicpi} As $G/A$ is abelian, $G'\subseteq A$. Let $p\in \pi(|G|)$.
If $p\nmid m$ then $AB_{p'}$ is a normal Hall $p'$-subgroup of $G$ and hence $p\in \pi$.
Suppose otherwise that $p\mid m$ and let $C$ be the unique subgroup of order $p$ in $A$.
Since $C$ is normal in $G$, it follows that $G'\setminus Z(G)$ has an element of order $p$ if and only if $A\setminus Z(G)$ has an element of order $p$ if and only if $C\not\subseteq Z(G)$. Since $\Aut(C)$ is cyclic of order $p-1$, if $p\in \pi$ and $N$ is a normal Hall $p'$-subgroup of $G$ then $G=N\rtimes P$ with $P$ a Sylow $p$-subgroup of $G$ containing $C$ and as $[P,C]=1$ it follows that $[G,C]\subseteq [N,C]\subseteq N\cap C=1$ and hence $C\subseteq Z(G)$. Conversely, if $C\subseteq Z(G)$ then $[A_p,A_{p'}B_{p'}]=1$ because the kernel of the restriction homomorphism $\Aut(A_p)\rightarrow \Aut(C)$ is a $p$-group. As $A_{p'}B$ normalizes $A_{p'}B_{p'}$ it follows that the latter is a normal Hall $p'$-subgroup of $G$ and hence $p\in \pi$.

\eqref{BasicG'} Let $p\in \pi'$, $c$ an element of order $p$ in $A$ and $a$ a generator of $A$.
Since $|\Aut(\GEN{c})|=p-1$ and $c\not\in Z(G)$,  we have that $a_p^b=a_p^k$ for some integer $k$ such that $\gcd(k,p)=1$. Moreover, $k-1$ is coprime with $p$ because $1\ne [c,b]=c^{k-1}$.
Then $A_p=\GEN{a_p^{k-1}}\subseteq G'$ and hence $A_p={G'}_p$.
Moreover, if $g\in A_p\cap B_p\setminus \{1\}$ then $[g,B]=1$ and $c\in \GEN{g}$, yielding a contradiction. Thus $A_p\cap B_p=1$.
Since this is true for each $p\in \pi'$, we have $A_{\pi'}={G'}_{\pi'}$ and $A_{\pi'}\cap B_{\pi'}=1$.

\eqref{BasicResume} is a direct consequence of \eqref{Basicpi} and \eqref{BasicG'}.

\eqref{BasicSemi} By \eqref{BasicHall} and \eqref{BasicG'}, $A_{\pi'}B_{\pi'}=A_{\pi'}\rtimes B_{\pi'}$ is the unique Hall $\pi'$-subgroup of $G$ and hence $G=(A_{\pi'}\rtimes B'_{\pi'}) \rtimes (A_\pi B_\pi)$.
Moreover, if $p\in \pi$ and $c$ is an element of order $p$ in $A_p$ then $c\in Z(G)$ by \eqref{Basicpi}. This implies that $[B_{p'},A_p]=1$ because the kernel of $\Res_p:\Aut(A_p)\rightarrow \Aut(\GEN{c})$ is a $p$-group. Then $[B_{\pi'},A_{\pi}B_{\pi}]=1$ and $A_{\pi}B_{\pi}=\prod_{p\in \pi} A_pB_p$.
\end{proof}

Next lemma shows that $\epsilon_G$ is determined by any minimal kernel of $G$.

\begin{lemma}\label{epsilonWD}
If $A$ is a minimal kernel of $G$ then $\epsilon_G=\epsilon_G(A)$.
\end{lemma}

	\begin{proof}
	Let $m=m_G=|A|$, $\epsilon=\epsilon_G(A)$ and $r=r_G(A)$.
	If $m_2\le 2$ then $\epsilon=1=\epsilon_G$. Otherwise $4\mid r_2$ and 
	$${G'}_2=\begin{cases} \GEN{a^{r_2}}, & \text{ if } \epsilon=1; \\
	\GEN{a^2}, & \text{ if } \epsilon=-1. \end{cases}$$
	Then
	$$|{G'}_2|=\begin{cases} \frac{m_2}{r_2}, & \text{ if } \epsilon=1; \\
	\frac{m_2}{2}, & \text{ if } \epsilon=-1; \end{cases}$$
	 and hence $\epsilon=-1$ if and only if $m_2=2|{G'}_2|>2$ if and only if $\epsilon_G=-1$.
	\end{proof}


Let
	$$R_G=\{r_G(A) : A \text{ is a minimal kernel of } G\}.$$
Next lemma shows that $|R_G|\le 2$ and in most cases $|R_G|=1$.

\begin{lemma}\label{Differentr}
Let $m=m_G$, $n=n_G$ and $o=o_G$.
Then the following statements are equivalent:
\begin{enumerate}
	\item\label{R>1} $|R_G|>1$.
	\item\label{RCond} $n_2\ge 4$, $m_2\ge 8$, $\epsilon_G=-1$, $o_2<n_2$ and $R_G=\{\frac{r}{2},r\}$ for some $r$ with $r_2=m_2$.
	\item\label{RCond2Some} $n_2\ge 4$, $m_2\ge 8$, $\epsilon_G=-1$, $o_2<n_2$, $r_2\in \{\frac{m_2}{2},m_2\}$ for some $r\in R_G$ and $[G:B]_2=\frac{m_2}{2}$ for some metacyclic factorization $G=AB$ with $m=|A|$.
	\item\label{RCond2All} $n_2\ge 4$, $m_2\ge 8$, $\epsilon_G=-1$, $o_2<n_2$, $r_2\in \{\frac{m_2}{2},m_2\}$ for some $r\in R_G$ and $[G:B]_2=\frac{m_2}{2}$ for every metacyclic factorization $G=AB$ with $m=|A|$.
\end{enumerate}

Furthermore, suppose that $G=AB$ is a metacyclic factorization satisfying the conditions of \eqref{RCond2Some} and let $a$ be a generator of $A$ and $b$ be a generator of $B$ and $s=[G:B]$. Let $C=\GEN{b^{\frac{nm_{2'}}{2s_{2'}}}a}$.
Then $G=CB$ is another metacyclic factorization with $|C|=m$ and $r_G(C)\ne r_G(A)$. 
\end{lemma}

\begin{proof}
Let $\epsilon=\epsilon_G$, $o=o_G$, $R=R_G$ and for every $p\in \pi$ let $R_p=\{r_p : r\in R\}$. 
Fix a minimal kernel $A$ of $G$ and let $r=r_G(A)$.

Let $p\in \pi$. 
If $\epsilon^{p-1}=1$ then $|{G'}_p|=\frac{m_p}{r_p}$.
Thus in this case $|R_p|=1$. 
Therefore $r_{2'}$ is constant for every $r\in R$ and hence $|R|=|R_2|$. 
Moreover, if $\epsilon=1$ then $G'_2=\frac{m_2}{r_2}$ and hence $R_2=\{\frac{m_2}{|{G'}_2|}\}$. In this case none of the conditions \eqref{R>1}-\eqref{RCond2All} hold.
Otherwise, $4\mid r_G(A)_2 \mid m_2$.
Thus, if $m_2<8$ then $r_G(A)_2=4$ for every minimal kernel $A$ of $G$ and hence $|R|=|R_2|=1$, so that again none of the conditions \eqref{R>1}-\eqref{RCond2All} hold.
Thus in the remainder of the proof we assume that $\epsilon=-1$ and $8\le m_2$. 
Then ${G'}_2=A^2$ and hence $\GEN{-1+r_G(A)_2}_{\frac{m_2}{2}} = \Res_{\frac{m_2}{2}}(\T_G(A))=\sigma_{{G'}_2}^{-1}(\Inn_G({G'}_2))$, which is independent of $A$.
This shows that if $R_2$ contains an element smaller than $\frac{m_2}{2}$ then it only has one element and hence again none of the conditions \eqref{R>1}-\eqref{RCond2All} hold.
So in the remainder of the proof we assume that $R_2 \subseteq \{\frac{m_2}{2},m_2\}$.

Suppose that $o_2=n_2$. Then, by \Cref{Basic}.\eqref{BasicResume}, $C_G({G'}_{\pi'})_2=A_2$, and hence $\GEN{-1+r_G(A)_2}_{m_2} = \Res_{m_2}(T_G(C_G({G'}_{\pi'})_2))$ is independent of $A$.
Therefore, in this case $|R_2|=1$, so that $|R|=1$.
So again in this case none of the conditions \eqref{R>1}-\eqref{RCond2All} hold and in the remainder of the proof we also assume that $o_2<n_2$.

Suppose that $n_2<4$. Then none of the condition \eqref{RCond}-\eqref{RCond2All} holds and as $\epsilon=-1$, we have $n_2=2$.
By means of contradiction suppose that \eqref{R>1} holds.
By the previous paragraph $R_2=\{\frac{m_2}{2},m_2\}$ and hence $G$ has two minimal kernels $A$ and $C$ with $r_G(A)_2=m_2$ and $r_G(C)_2=\frac{m_2}{2}$.
If $G=AB$ and $G=CD$ are metacyclic factorization of $G$ then $A_2B_2$ and $C_2D_2$ are Sylow $2$-subgroups of $G$ and hence they are isomorphic.
However, by \Cref{ParametersTDelta}.\eqref{mDividernrs}, $[A_2B_2:B_2]$ is either $m_2$ or $\frac{m_2}{2}$. In the first case $A_2B_2$ is dihedral and in the second case $A_2B_2$ is quaternionic. This yields a contradiction because from $r_G(C)_2=\frac{m_2}{2}$ it follows that $C_2D_2$ is neither dihedral nor quaternionic.

Thus in the remainder we assume that $m_2\ge 8$, $n_2\ge 4$, $o_2<n_2$, $\epsilon=-1$ and $R_2\subseteq \{\frac{m_2}{2},m_2\}$. Moreover, by the above arguments we have that $R\subseteq \{\frac{r}{2},r\}$ for some $r$ with $r_2=m_2$. Thus \eqref{R>1} and \eqref{RCond} are equivalent.

\eqref{RCond2All} implies \eqref{RCond2Some} is clear.

\eqref{RCond2Some} implies \eqref{RCond}.
Let $G=AB$ be a metacyclic factorization of $G$ satisfying the conditions of \eqref{RCond2Some}. Let $s=[G:B]$ and $r=r_G(A)$. 
Select generators $a$ of $A$ and $b$ of $B$ and let $z=b^{\frac{nm_{2'}}{2s_{2'}}}$, $c=za$ and $C=\GEN{c}$.
We will prove that if $G=CB$ is another metacyclic factorization with $|C|=m$ and $r_G(C)\ne r$, so that \eqref{RCond} holds.
Indeed, since $o_2<n_2$, we have $[z,a_{\pi'}]=1$. Moreover, $[z_{p'},a_p]=1$ for every $p\in \pi$. If moreover, $p\ne 2$ then $[z_p,a_p]=1$ because $[b^n,a]=1$. Finally, $r_2\in \{\frac{m_2}{2},m_2\}$ and hence $o_{m_2}(-1+r_2)=2$. As $4\mid n$ and $a_2^{b_2}=a_2^{-1+r_2}$ it follows that $[z_2,a_2]=1$.
This shows that $z\in Z(G)$. 
As $s=[G:B]$ and $[G:A]=n$ we have $b^n=a^{sx}$ for some integer $x$ coprime with $m$. 
Then $c^2 = a^{2+sx\frac{m_{2'}}{s_{2'}}} = a^{2+xs_2m_{2'}}=a^{2+x\frac{m}{2}}=a^{2+\frac{m}{2}}$. 
As $8\mid m$ it follows that $|C|=m$. Suppose that $a^b=a^t$. Then $t+1\equiv r_2 \mod m_2$. Let $r'\in \N$ with $r'_{2'}=r_{2'}$ and $\{r_2,r'_2\}=\{\frac{m_2}{2},m_2\}$ and let $t'$ be an integer such that $t'\equiv t \mod m_{2'}$ and $t'\equiv -1+r'_2 \mod m_2$. 
As $8\mid m$ we have $t'\equiv t \equiv -1 \mod 4$ and hence $t'=1+2y$ for some odd integer $y$. 
Then $c^{t'}=zz^{t'-1}a^{t'}=zz^{2y}a^{t'}=za^{t'+y\frac{m}{2}}$. 
Moreover, $t'+y\frac{m}{2} \equiv t' \equiv t \mod m_{2'}$ and $t'+y\frac{m}{2}\equiv -1+r'_2+\frac{m_2}{2} \equiv -1+r_2 \equiv t \mod m_2$. Therefore $c^{t'}=za^t=c^b$. This shows that $C$ is a cyclic normal subgroup of $G$ and clearly $G=CB$ is a metacyclic factorization satisfying the desired condition. 

Before proving \eqref{R>1} implies \eqref{RCond2All} we prove that if $G=AB=CD$ are metacyclic factorizations with $|A|=|B|=m$ then $[G:B]_2=[G:D]_2$.
The assumption $\epsilon=-1$ implies that ${G'}_2=A^2=C^2$.
As $A_2B_2$ and $C_2D_2$ are Sylow $2$-groups of $G$ we may assume that they are equal and hence if $A_2=\GEN{a}$ and $B=\GEN{b}$ we may write $c=b^ia^j$ and $d=b^ka^l$. Since $c^2\in C^2=A^2$ we have $\frac{n_2}{2}\mid i$ and as $4\mid n$, necessarily $2\mid i$ and hence $2\nmid k$.
Then, using that $r_G(A),r_G(C)\in \{\frac{m_2}{2},m_2\}$ we have that $d^2 = b^{2k}$ or $d^2=b^{2k}a^{l\frac{m_2}{2}}$. In both cases $d^4=b^4$ and hence $D^4=B^4$. As $4\mid n$ it follows that $A_2\cap B_2=B_2^{n_2}=D_2^{n_2}=C_2\cap D_2$.
Therefore, $[G:B]_2=[A_2B_2:B_2]=[A_2,A_2\cap B_2]=[C_2:C_2\cap D_2]=[G,D]_2$, as desired.

\eqref{R>1} implies \eqref{RCond2All}. Suppose that $|R|>1$.
By the assumptions and the previous arguments we know that the only condition from \eqref{RCond2All} which is not clear is that if $G=AB$ is a metacyclic factorization with $m=|A|$ and $s=[G:B]$ then $s_2=\frac{m_2}{2}$. So suppose that $s_2=m_2$. Since $|R|>1$, there is a second metacyclic factorization $G=CD$ with $|C|=m$ and $\{r_G(A)_2,r_G(C)_2\}=\{\frac{m_2}{2},m_2\}$. By the previous paragraph $[G:D]_2=[G:B]_2=1$.
By symmetry we may assume that $r_G(A)_2=m_2$ and $r_G(C)=\frac{m_2}{2}$.
As above we may assume that $A_2B_2=C_2D_2$ and if $A_2=\GEN{a}$, $B_2=\GEN{b}$, $C_2=\GEN{c}$ and $D_2=\GEN{d}$ then $a^b=a^{-1}$, $c^d=c^{-1+\frac{m_2}{2}}$,
${G'}_2=A_2^2=C_2^2$, $A_2\ne C_2$ and $A_2\cap B_2=C_2\cap D_2=1$.
Write $c=b^ia^j$ and $d=b^ka^l$ with $i,j,k,l\in \N$.
Since $c^2\in A$ we have that  $\frac{n_2}{2}\mid i$ and as $4\mid n_2$, we have that $k$ is odd and $[b^i,a]=1$. Thus $b^{2i}=c^2a^{-2j}\in A_2\cap B_2=1$.
Then $c^2=a^{2j}$ and as $C^2=A^2$, necessarily $j$ is odd.
However, from $b^{2i}=1$, $[b^i,a]=1$ and $8\mid m$ we have
$b_2^ia_2^{(-1+\frac{m_2}{2})j}=b_2^{(-1+\frac{m_2}{2})i}a_2^{(-1+\frac{m_2}{2})j}= c_2^{-1+\frac{m_2}{2}} = c_2^{d} = b_2^ia_2^{-j}$ and hence $2\mid j$, a contradiction.
\end{proof}

In our next result we show a way to decide if a factorization of $G$ is minimal and we prove that the following algorithm transforms a metacyclic factorization of $G$ into a minimal one.

\begin{algorithm}\label{AlgoMinMF}
	{\sc Input}: A metacyclic factorization $G=AB$ of a finite group $G$.
	
	{\sc Output}: $a,b\in G$ with $G=\GEN{a}\GEN{b}$ a minimal metacyclic factorization of $G$. 
	
	\begin{enumerate}
		\item $m:=|A|$, $n:=[G:A]$, $s:=[G:B]$, 
		\item\label{AlgoGenerators} $a:=$ some generator of $A$, $b:=$ some generator of $B$, and $y\in \N$ with $b^n=a^y$.
		\item $r:=r_G(A)$, $\epsilon:=\epsilon_G(A)$ and $o=o_G(A)$.
		\item\label{AlgomG} for $p\in \pi(r)$ with $\epsilon^{p-1}=1$ 
		\begin{enumerate}
			\item\label{Algo11}  if $s_p\nmid n$ then $b:=ba_p$ and $s:=s_{p'}n_p$.
			\item\label{Algomp}  if $r_p\nmid s$, $s_po_p\mid n$ and $t\in \N$  satisfy $a_p^{b_p}=a_p^t$, compute $x\in \N$ satisfying  $x\Ese{t^{\frac{n}{s_p}}}{s_p}\equiv r-y\mod m_p$ and set
			$a:=b_p^{\frac{n}{s_p}} a_{p'} a_p^x$, $m:=s_p\frac{m}{r_p}$, $n:=n\frac{r_p}{s_p}$, and $$(r,\epsilon):=\begin{cases} (4r_{2'},-1), &  \text{ if } 8\mid m, s_p=2, \text{ and } r_2=\frac{m_2}{2}; \\(r_{p'}s_p,1), & \text{otherwise}.\end{cases}$$
		\end{enumerate}
		\item\label{AlgorG} If $\epsilon=-1$, $4\mid n$, $8\mid m$, $o_2<n_2$ and $r_2\nmid s$ then $a:=b^{\frac{m_{2'}n}{2s_{2'}}}a$ and $r:=r_{2'}s_2$
		\item If $\epsilon=-1$ and $s_2=r_2n_2$ then $b:=ba_2$ and $s:=\frac{s}{2}$.
		\item  Return $(a,b)$.
	\end{enumerate}
\end{algorithm}

\begin{proposition}\label{MiniMF}
Let $G=AB$ be a metacyclic factorization and let $m=|A|$, $n=[G:A]$, $s=[G:B]$ and $T=T_G(A)$.  
Then $G=AB$ is minimal as metacyclic factorization of $G$ if and only if $T$ is $(n,s)$-canonical.

Furthermore, if the input of \Cref{AlgoMinMF} is a metacyclic factorization of $G$ and its output is $(a,b)$ then $G=\GEN{a}\GEN{b}$ is a minimal metacyclic factorization of $G$.
\end{proposition}

\begin{proof}
Let $(r,\epsilon,o)=\inv{T_G(A)}$. 
By \Cref{Basic}, $\pi'=\pi(m)\setminus \pi(r)$.
Fix $y,t\in \N$ with $b^n=a^y$ and $a^b=a^t$. 
Then $s=\gcd(t,m)$, $\gcd(t,m)=1$, $r_{2'}=\gcd(m_{2'},t-1)$ and $r_2=\gcd(m_2,t-\epsilon)$.
For every prime $p$ let $G_p=A_pB_p$.
\medskip

\noindent\textbf{Claim 1}. If condition (Can+) holds then $A$ is a minimal kernel of $G$.

Suppose that condition (Can+) holds and let $C$ be kernel of $G$. We want to prove that $|C|\ge m$ and for that it is enough to show that $|C_p|\ge m_p$ for every prime $p$. 
This is obvious if $m_p=1$, and it is a consequence of \Cref{Basic}.\eqref{BasicG'}, if $p\in \pi'$.
So we suppose that $p\in \pi$ and $m_p\ne 1$. Hence $p\mid r$.

Suppose first that $\epsilon^{p-1}=-1$.
Then $p=2$ and $A_2^2=G_2'\subseteq C_2$. However $C_2\not\subseteq A_2^2$ because $G_2/A_2^2$ is not cyclic. Therefore $|C_2|\ge 2|A_2^2|=m_2$.

Suppose otherwise that $\epsilon^{p-1}=1$.
Then $G_p'=A_p^{r_p}$ and $|{G'}_p|=\frac{m_p}{r_p}$.
Assume that $r_p\mid s_p$.
Then $G_p/G_p'=(A_p/{G'_p})\times (B_p{G'}_p/{G'}_p)$ and $r_p=|A_p/{G'}_p|\le n_p=[B_p{G'}_p:{G'}_p]$.
As $(G_p/{G'}_p)/(C_p/{G'}_p)\cong G_p/C_p$ is cyclic, necessarily $r_p\mid [C_p:{G'}_p]$ and hence $m_p\mid |C_p|$, as desired.
Assume otherwise that $r_p\nmid s_p$.
By condition (Can+) we have $s_p\mid n_p$ and $s_po_p\nmid n_p$.
In particular $p\mid o_p$.
By \Cref{Basic}.\eqref{BasicG'}, $C_{\pi'}=A_{\pi'}$ and thence $C_p\subseteq C_{G_p}(A_{\pi'})_p=A_pB_p^{o_p}$.
Using again that $G_p/C_p$ is cyclic and $p\mid o_p$, we must have $C_p=\GEN{b_p^xa_p}$ for $x\in \N$ with $o_p\mid x$ and $x\le n$.
Let $R\in \N$ such that $a_p^{b_p^x}=a_p^R$. Then $R$ satisfies the hypothesis of \Cref{PropEse}.\eqref{vpEspecial+} and hence $v_p\left(\Ese{R}{\frac{n}{x_p}}\right)=v_p(n)-v_p(x)\le v_p(n)-v_p(o)<v_p(s)=v_p(yx_{p'})$ and therefore $v_p\left(yx_{p'}+\Ese{R}{\frac{n}{x_p}}\right)=v_p(n)-v_p(x)$.
Then $|C_p|=\frac{n_p}{x_p}|(b_p^xa_p)^{\frac{n_p}{x_p}}| =
\frac{n_p}{x_p} \left| a_p^{yx_{p'}+\Ese{R}{\frac{n_p}{x_p}}} \right|=m_p$.
This finishes the proof of Claim 1.
\medskip

\noindent\textbf{Claim 2}. If $T_G(A)$ is $(n,s)$-canonical then for every metacyclic factorization $G=CD$ with $|C|=m$ one has $r_G(C)\ge r$ and $|D|\le |B|$.

If $r_G(C)<r$ then, by \Cref{Differentr}, $m_2\ge 8$, $n_2\ge 4$, $\epsilon=-1$, $o_2<n_2$, $r_G(C)_2=\frac{m_2}{2}=s_2$ and $r_2=m_2$, in contradiction with the second part of condition (Can--). Thus $r_G(C)\ge r$.

To prove that $|D|\le |B|$ we show that $|D_p|\le |B_p|$ for each prime $p$. 
This is clear if $p\nmid m$ and a consequence of \Cref{Basic}.\eqref{BasicResume} if $p\in \pi'$. 
Otherwise $p\mid r$. 
Since both $G_p$ and $C_pB_p$ are Sylow $p$-subgroups of $G$ we may assume that $G_p=C_pD_p$.

Assume first that $\epsilon^{p-1}=1$. Then by assumption $s_p\mid n_p$. Let $d=b_p^xa_p^y$ be a generator of $D_p$ and let $R\in \N$ such that $a_p^{b_p^x}=a_p^R$. 
The assumption $\epsilon^{p-1}=1$ implies that $R$ satisfies the hypothesis of 
\Cref{PropEse}.\eqref{vpGeneral} and hence $m_p\mid \Ese{R}{m_p\frac{n_p}{s_p}}$ and from \eqref{Potencia} we deduce that $d^{\frac{m_pn_p}{s_p}} = a_p^{y\Ese{(1+r_p)^x}{m_p\frac{n_p}{s_p}}}=1$ and hence $|D_p|\le \frac{m_pn_p}{s_p}=|b_p|$.
Suppose otherwise that $\epsilon^{p-1}=-1$, i.e. $p=2$ and $\epsilon=-1$.
Then $C_2^2={G'}_2=A_2$ and $C_2\cap D_2\subseteq Z(G_2)\cap C_2 = Z(G_2)\cap C_2^2 = Z(G_2)A=A^{\frac{m_2}{2}}$ and hence  $|C_2\cap D_2|\le 2$.
Thus $|D_2|=[D_2:C_2\cap D_2]\;|C_2\cap D_2| = [G_2:C_2]\;|C_2\cap D_2|\in \{n_2,2n_2\}$.
Similarly, $|B_2| \in \{n_2,2n_2\}$.
If $|B_2|=2n_2$ then $|D_2|$ divides $|B_2|$ as desired.
Suppose otherwise that $|B_2|=n_2$.
Then $m_2=s_2$ and hence $m_2$ divides $\frac{r_2n_2}{2}$, by the hypothesis (Can--) and \Cref{ParametersTDelta}.\eqref{mDividernrs}.
If $D_2\subseteq \GEN{a,b_2^2}$ then $C_2=\GEN{b_2a_2^x}$ for some integer $x$ and hence $n_2=2$ because $C_2^2=\GEN{a_2^2}$. Then $D_2\subseteq \GEN{a_2}$ so that $D_2$ is normal in $G_2$ and hence $\GEN{a_2^2}=C_2^2=[D_2,C_2]\subseteq C_2\cap D_2 \subseteq \GEN{a_2^{\frac{m_2}{2}}}\subseteq \GEN{a_2^2}$.
Then $m_2=4$ and $G_2$ is dihedral of order $8$.
Then every metacyclic factorization of $G_2$ is of the form $\GEN{a_2}\GEN{c}$ with $|c|=2$.
Thus $|D_2|=2=|b_2|$, as wanted.
Assume otherwise that $D_2\not\subseteq \GEN{a_2,b_2^2}$. Then $D_2=\GEN{b_2a_2^x}$ for some integer $x$ and let $R\in \N$ such that $a_2^{b_2}=a_2^R$. The hypothesis $\epsilon=-1$ implies that $R$ satisfies the hypothesis of \Cref{PropEse}.\eqref{vpEspecial}. 
Since $m_2$ divides $\frac{r_2n_2}{2}$, we get $v_2(\Ese{R}{n_2})=v_2(r_2)+v_2(n_2)-1\ge v_2(m_2)$ and hence $(b_2a_2^x)^{n_2}=a_2^{x\Ese{-1+r_2}{n_2}}=1$. Then $|D_2|=n_2$, as desired.
This finishes the proof of Claim 2.\medskip

The necessary part in the first statement of the proposition follows from claims 1 and 2. 
\medskip

\noindent\textbf{Claim 3}. If $p\mid r$, $\epsilon^{p-1}=1$ and $s_p\nmid n_p$ then $[G:ba_p]=s_{p'}n_p<s$. 

First of all  $n= |ba_pA|$ and hence $n$ divides $|ba_p|$. 
Using \eqref{Potencia} we have $(ba_p)^n=a_{p'}^ya_p^{y+\Ese{t}{n}}$ and $v_p([G:\GEN{ba_p}])=v_p(\Ese{t}{n})=v_p(n)<v_p(s)=v_p(y)$, by \Cref{PropEse}.\eqref{vpGeneral} and the assumption.
Thus $|ba_p|=n\frac{m}{s_{p'}n_p}$ and hence $[G:ba_p]=s_{p'}n_p$. This finishes the proof of Claim 3.\medskip

By Claim 3, if the first part of (Can+) fails then $G=AB$ is not minimal because $G=A\GEN{ba_p}$ is a factorization with $[G:b]>[G:\GEN{ba_p}]$. 
Moreover, the factorization $G=A\GEN{ba_p}$ satisfies the first part of condition (Can+) and hence after step \eqref{Algo11} of \Cref{AlgoMinMF}, the factorization $G=\GEN{a}\GEN{b}$ satisfies the first part of (Can+) for the prime $p$.

\noindent\textbf{Claim 4}. Suppose that $p\mid r$, $\epsilon^{p-1}=1$, $s_p\mid n$, $r_p\nmid s$ and $s_po_p\mid n$.
Let $R\in \N$ with $a_p^{b_p^{\frac{n}{s_p}}}=a^R$.
Then there is an integer $x$ such that $r-y\equiv x\Ese{R}{s_p} \mod m_p$.
This justify the existence of $x$ in step \eqref{AlgomG} of \Cref{AlgoMinMF}.
Let $c=b_p^{\frac{n}{s_p}}a_{p'}a_p^x$ and $C=\GEN{c}$. Then $G=CB$ is a metacyclic factorization of $G$ with $|C|=m\frac{s_p}{r_p}<|A|$. Moreover, $$(r_G(C),\epsilon_G(C)):=\begin{cases} (4r_{2'},-1), &  \text{ if } 8\mid m, s_p=2, \text{ and } r_2=\frac{m_2}{2}; \\(r_{p'}s_p,1), & \text{otherwise}.\end{cases}$$

The assumption $s_po_p\mid n_p$ implies that $o_p\mid \frac{n}{s_p}$ and hence $[b_p^{\frac{n}{s_p}},a_{\pi'}]=1$. As also $[b_p,a_{\pi\setminus \{p\}}]=1$ we deduce that $[b_p^{\frac{n}{s_p}},a_{p'}]=1$.
On the other hand, since $r_p\nmid s_p$,  $v_p(y)=v_p(s)<v_p(r)$ and therefore $v_p(r-y)=v_p(s)=v_p(\Ese{t}{s_p})$, by \Cref{PropEse}.\eqref{vpGeneral}.
Therefore there is an integer $x$ coprime with $p$ such that $r-y\equiv x\Ese{R}{s_p} \mod m_p$.
Using \eqref{Potencia} we have $c^{s_p}=b_p^{n} a_{p'}^{s_p} a_p^{x\Ese{R}{s_p}}=a_{p'}^{s_p} a_p^{y+x\Ese{R}{s_p}}=a_{p'}^{s_p} a_p^r$.
Then ${G'}_{p'}\subseteq \GEN{a_{p'}}\subseteq C$ and ${G'}_p=\GEN{a_p^r}\subseteq C$. Thus $G'\subseteq C$ and hence $G=CB$ is a metacyclic factorization of $G$ with $|C|=s_p|a_{p'}||a_p^r|=m\frac{s_p}{r_p}<m=|A|$. 
As $C_{p'}=A_{p'}$, we have $r_G(C)_{p'}=r_G(A)_{p'}=r_{\pi'}$. 
If $\epsilon_G(C)^{p-1}=1$ then $\frac{m_p}{r_p}=|G'_p|=\frac{|C_p|}{r_G(C)_p}=\frac{m_ps_p}{r_pr_G(C)_p}$ and hence in this case $r_G(C)=r_{p'}s_p$. 
Otherwise, i.e. if $p=2$ and $\epsilon_G(C)=-1$ then $2|C_2|\le s_2\le |C_2|$ and $4\le r_G(C)_2\le |C_2|=\frac{m_2s_2}{r_2}=2|{G'}_2|=\frac{2m_2}{r_2}$ and hence $s_2=2$, $|C_2|=4=r_G(C)_2$ and $r_2=\frac{m_2}{2}$.  
Conversely, if $s_2=2$ and $r_2=\frac{m_2}{2}$ then $|C_2|=4$ and hence $r_G(C)_2=4$. Moreover, as $G_2$ is not commutative then $\epsilon_G(C)=-1$. 
This finishes the proof of Claim 4.\medskip 

Claim 4 shows that if the first part of (Can+) holds but the second one fails then $G=AB$ is not minimal. It furthermore the parameters associated to the factorization $G=CB$, i.e. $|C|, [G:C], [G:B], r_G(C), \epsilon_G(C), o_G(C)$, satisfy condition (Can+) for the prime $p$ and hence, after step \eqref{Algomp} of \Cref{AlgoMinMF}, the current factorization $G=\GEN{a}\GEN{b}$ satisfies this condition. Moreover, if $\epsilon_G(C)=1$ then $r_p(C)=s_p\le n_p$ and condition (C+) holds for the prime $p$.
Thus when the algorithm finishes the loop in step \eqref{AlgomG}, the metacyclic factorization satisfies condition (Can+) and hence the current value of $\GEN{a}$ is a minimal kernel of $G$ by Claim 1. 

Observe that the modification of $a$ and $b$ in steps \eqref{Algo11} and \eqref{Algomp} for some prime $p$ does not affect the subsequent calculations inside the loop. Indeed, suppose that $p$ and $q$ are two different divisors of $r$ with $\epsilon^{p-1}=\epsilon^{q-1}=1$, and the prime $p$ has been considered before the prime $q$ in step \eqref{AlgomG}. This has affected $a$ and $b$ which have been transformed by first transforming $b$ into $d=ba_p$ and then transforming $a$ into $c=d_pa_{p'}a_p^x=b_pa_{p'}a_p^{1+x}$.
In principal we should recalculate the natural number $y$ computed in step \eqref{AlgoGenerators} to a new $y'$. However, as $p\in \pi$, $[b_{p'},a_p]=[b_{q'},a_p]=1$ and hence $a_{p'}=c_{p'}$ and $b_{p'}=d_{p'}$. Therefore $d_q=c_q^y$ and hence $y'\equiv y \mod m_q$. Therefore when in step \eqref{Algomp} for the prime $q$ we compute $x$ satisfying if $r-y\equiv x \Ese{R}{s_q}\equiv \mod m_q$ we also have $r-y'\equiv x \Ese{R}{s_q} \mod m_q$.

By \Cref{Differentr}, if the second part of condition (Can--) is satisfied then $r_G(A)=r_G$. Otherwise, $r_G(A)>r_G$,  and hence the factorization $G=AB$ is not minimal, However, after step \eqref{AlgorG} the factorization $G=\GEN{a}\GEN{b}$ satisfy both $|a|=m_G$ and $r_G(\GEN{a})=r_G$. In the remainder of the algorithm the kernel $\GEN{a}$ is not modified and hence this is going to be valid in the remainder of the algorithm.

Finally suppose that the first part of (Can--) fails, so that $p=2$, $\epsilon=-1$ and $s_2=r_2n_2$.
Then $4\mid r$ and $\GEN{t}_{m_2}=\GEN{-1+r_2}_{m_2}$. 
Moreover, by \Cref{ParametersTDelta}.\eqref{mDividernrs}, we have that $s_2\in \{\frac{m_2}{2},m_2\}$ and $m_2\mid r_2n_2$. Therefore $s_2=m_2=r_2n_2$.
Then  $v_2(\Ese{t}{n_2})=v_2(r)+v_2(n)-1=v_2(m)-1$, by \Cref{PropEse}.\eqref{vpEspecial}.
As in the proof of Claim 3, we use the metacyclic factorization of  $G=A\GEN{ba_2}$. If $G=AB$ is minimal then we have $n|(ba_2)^n|=|ba_2|\le |b|=n|a^s|=n\frac{m}{s}$. Therefore $|(ba_2)^n|\le \frac{m}{s}$.
Using \eqref{Potencia} once more and $[b_{2'},a_2]=1$, we obtain $(ba_2)^n=a^ya_2^{\Ese{t}{n_2}}=a_{2'}^ya_2^{\frac{m_2}{2}}$.
Thus $|(ba_2)^n|=2\frac{m}{s}$ and hence $|ba_2|=2\frac{ms}{s}=2|B|$, contradicting the minimality. Thus $G=AB$ is not minimal. 
	Moreover, the new metacyclic factorization satisfies (Can--) because, $|ba_2|_2=2|b|_2$ and hence if $s'=[G:\GEN{ba_2}]$ then $s'_2=\frac{m_2}{2}\ne m_2=r_2n_2$.
\end{proof}

%
%

In order to prove that the last entry of $\INV(G)$ is well defined and prove \Cref{Invariants} we need one more lemma which is inspired in Lemmas 5.5 and 5.7 of \cite{Hempel2000}.

\begin{lemma}\label{cdw}
Let $p$ be a prime and consider the group $P=\G_{m,n,s,\epsilon+r}$ with $m$ and $n$ powers of $p$, $r$ and $s$ divisors of $m$ and $\epsilon\in \{1,-1\}$ satisfying the following conditions: $p\mid r$, $m\mid rn$, if $4\mid m$ then $4\mid r$, if $\epsilon=1$ then $m\mid rs$ and if $\epsilon=-1$ then $2\mid n$, $4\mid m$ and $m\mid 2s$. 
Let $o$ be a divisor of $n$ and $N=\GEN{a,b^o}$.
	Denote
	$$w=\begin{cases} \min(o,\frac{m}{r},\max(1,\frac{s}{r},\frac{so}{n})), & \text{if } \epsilon=1; \\
	1, & \text{if } \epsilon=-1 \text{ and }, o\mid 2  \text{ or } m\mid 2r; \\
	\frac{m}{2r}, & \text{if } \epsilon=-1, 4\mid o<n, 4r\mid m,  \text{ and  if } s\ne nr \text{ then } 2s=m<nr; \\
	\frac{m}{r}, & \text{otherwise}.	
	\end{cases}$$
	
	If $y$ is an integer coprime with $p$ then the following conditions are equivalent:
	\begin{enumerate}
		\item There are $c\in N$ and $d\in b^yN$ such that $P=\GEN{c,d}$, $|c|=m$, $d^n=c^s$ and $c^d=c^{\epsilon+r}$.
		\item $y\equiv 1 \mod w$.
	\end{enumerate}
\end{lemma}

\begin{proof}
Observe that $N$ is the unique subgroup of $G$ of index $o$ containing $a$. 
	We will make a wide use of \eqref{Potencia} and \Cref{PropEse}, sometimes without specific mention. 
	We consider separately the cases $\epsilon=1$ and $\epsilon=-1$.

	\textbf{Case 1}. Suppose $\epsilon=1$. 
	
	(1) implies (2).
	Suppose that $c$ and $d$ satisfy the conditions of (1). 
	If $w=1$ then obviously (2) holds. So we may assume that $w\ne 1$ and in particular $p\mid o$ and $pr\mid m$.
	The first implies that $N \subseteq \GEN{a,b^p}$ and the second that $P/\GEN{a^p,b^p}$ is not cyclic. 
	Therefore $c\not\in \GEN{a^p,b^p}$ and hence $\GEN{c}=\GEN{b^{xv}a}$ with $o\mid v\mid n$ and $p\nmid x$.
	Write $d=b^{y_1}a^z$ with $y_1,z\in \Z$. 
	From the assumption $d\in b^yN$ we have that $y_1\equiv y \mod o$ and hence $y\equiv y_1\mod w$. 
	Therefore, it suffices to prove that $y_1\equiv 1 \mod w$.
	From $c^d=c^{1+r}$ we have
	$$b^{xv} a^{z(1-(1+r)^{xv})+(1+r)^{y_1}} = (b^{xv}a)^{b^{y_1}a^z} = (b^{xv}a)^{1+r} = b^{xv}a b^{xvr} a^{\Ese{(1+r)^{xv}}{r}}.$$
	Then $n\mid vr$ and $b^{xvr} = a^{xs\frac{vr}{n}}$. Thus
	$$z(1-(1+r)^{xv})+(1+r)^{y_1}-1 \equiv xs\frac{vr}{n}+\Ese{(1+r)^{xv}}{r} \mod m.$$
	This implies that that $r$ divides $xs\frac{vr}{n}$, since $r$ divides $m$. 
	As $r$ is coprime with $x$, it follows that $n$ divides $sv$. 
	Moreover, $(1+r)^{xv} \equiv 1 \mod rv$, by \Cref{PropEse}.\eqref{vpGeneral}, and hence 
	$\Ese{(1+r)^{xv}}{r}\equiv r \mod rv$. As $r,v,m$ and $s$ are powers of $p$ we deduce that 
	$$(1+r)^{y_1} \equiv 1 +r \mod \min(m,rv,\frac{svr}{n}).$$
	Using \Cref{PropEse}.\eqref{opGeneral} it follows that $y_1\equiv 1 \mod \min(\frac{m}{r},v,\frac{sv}{n})$. 
	
	Suppose that $y_1\not\equiv 1 \mod w$. Then 
	$$\min\left(\frac{m}{r},o,\frac{so}{n}\right) \le \min\left(\frac{m}{r},v,\frac{sv}{n}\right) < w = \min\left(\frac{m}{r},o,\max\left(1,\frac{s}{r},\frac{so}{n}\right)\right)$$ 
	and hence $\frac{s}{r}>\left(1,\frac{so}{n}\right)$ and $\frac{m}{r}\ge w =\min\left(o,\frac{s}{r}\right)>\min\left(\frac{m}{r},v,\frac{sv}{n}\right)$. 
	Thus  
	$$\frac{s}{r}\ge w = \min\left(o,\frac{s}{r}\right) >  \min\left(v,\frac{sv}{n}\right)\ge \min\left(o,\frac{so}{n}\right).$$
	Since $n\mid vr$ it follows that $\min(v,\frac{sv}{n})<\frac{s}{r}\le \frac{sv}{n}$ and hence $o\le v = \min(v,\frac{vs}{n})<\min(o,\frac{s}{r})$, a contradiction. 
	
	(2) implies (1). 
	We now suppose that $y\equiv 1 \mod w$ and we have to show that there is $c\in N$ and $d\in b^yN$ satisfying the conditions in (1). 
	If $y\equiv 1 \mod o$ then $bN=b^yN$ and hence $c=a$ and $d=b$ satisfy the desired condition. 
	If $(1+r)^y\equiv 1 +r \mod m$ then $a^{b^y}=a^{1+r}$ and hence $c=a^y$ and $b^y$ satisfy the desired conditions. 
	So we suppose that $y\not\equiv 1 \mod o$ and 
	$(1+r)^y\not\equiv 1+r \mod m$. The first implies that $w<o$ and the second that $y-1$ is not multiple of $o_m(1+r)=\frac{m}{r}$, by \Cref{PropEse}.\eqref{opGeneral} and hence $w<\frac{m}{r}$. 
	Thus $w=\max(1,\frac{s}{r},\frac{os}{n})<\min(o,\frac{m}{r})$.
	
	By \Cref{PropEse}.\eqref{opGeneral} we have $(1+r)^y = 1+r(1+xu)$ with $p\nmid x$, $u$ a power of $p$ and $v_p(w)\le v_p(u)=v_p(y-1)<v_p(\frac{m}{r})\le v_p(s)$. Moreover, if $u=1$ then $p\nmid 1+x$. 
	Let $c_1=b^{x\frac{nu}{s}}a$. 
	We now prove that $|c_1|=m$. 
	Observe that $\frac{nu}{s}\ge \frac{nw}{s} \ge o$. 
	Therefore $c_1\in N$.
	Moreover, as $v_p(u)<v_p(s)$ it follows that $|c_1\GEN{a}|=\frac{s}{u}$ and $c_1^{\frac{s}{u}}=a^{xs+\Ese{(1+r)^{x\frac{nu}{s}}}{\frac{s}{u}}}$. 
	If $u\ne 1$ then $v_p(r)\ge v_p(\frac{s}{w})\ge v_p(\frac{s}{u}) = v_p(\Ese{(1+r)^{x\frac{nu}{s}}}{\frac{s}{u}})= v_p(xs+\Ese{(1+r)^{x\frac{nu}{s}}}{\frac{s}{u}})$ and therefore $G'=\GEN{a^r}\subseteq \GEN{c_1}$ and $|c_1|=m$, as desired. 
	Otherwise, i.e. if $u=1$ then $w=1$ and hence $s\le r$ and $p\mid o\mid \frac{n}{s}$. Then $xs+\Ese{(1+r)^{x\frac{nu}{s}}}{s}\equiv s(x+1) \not\equiv 0 \mod pr$ because $s\le r$ and $p\nmid x+1$. Therefore also in this case $v_p(r)\le v_p(xs+\Ese{(1+r)^{x\frac{nu}{s}}}{s})$ and hence $G'\subseteq \GEN{c_1}$ and $|c_1|=m$, as desired. 
	
Since $(1+r)^{x\frac{nu}{s}} \equiv 1 \mod r\frac{nu}{s}$ we have $\Ese{(1+r)^{x\frac{nu}{s}}}{r}\equiv r \mod r\frac{nu}{s}$. 
Therefore $(1+r)^y-1-xru-\Ese{(1+r)^{x\frac{nu}{s}}}{r} \equiv 0 \mod \frac{rnu}{s}$. 
Moreover, $v_p(1-(1+r)^{x\frac{nu}{s}})=v_p(r\frac{nu}{s})$, and hence there is an integer  $z$ satisfying
	$$z(1-(1+r)^{x\frac{nu}{s}})+(1+r)^y \equiv 1 + xru+\Ese{(1+r)^{xu}}{r} \mod m.$$
Let $d=b^ya^z\in b^yN$. 
Using that $u\ge w\ge \frac{s}{r}$ we have 
$$c_1^d = (b^{x\frac{nu}{s}}a)^{b^ya^z} 
= b^{x\frac{nu}{s}}a^{z(1-(1+r)^{x\frac{nu}{s}})+(1+r)^y} = b^{x\frac{nu}{s}}a^{1+xru+\Ese{(1+r)^{x\frac{nu}{s}}}{r}} = 
c_1^{1+r},$$
On the other hand
	$$d^n = (b^ya^z)^n = 
	a^{sy +z\Ese{(1+r)^y}{n}}$$
and 
$$c_1^s = (b^{x\frac{nu}{s}}a)^s = a^{xus+\Ese{(1+r)^{x\frac{nu}{s}}}{s}}.$$
if $s\ge n$ then  $o>w=\max(\frac{so}{n},\frac{s}{r})\ge \frac{so}{n}\ge o$, a contradiction. 
Therefore, $s$ is a proper divisor of $n$ and hence  $v_p(sy+z\Ese{(1+r)^y}{n})=s$. 
Then $d^n$ and $c_1^s$ are elements of $\GEN{a}$ of the same order. Therefore $b^n=c^{ks}$ for some integer $k$ coprime with $p$. Then $c=c_1^k$ and $d$ satisfy the conditions of (1).

	\textbf{Case 2}. Suppose that $\epsilon=-1$. 
	
	(1) implies (2). 
	Suppose that $c$ and $d=b^ya^z$ satisfy the conditions of (1). 
	Then $4\mid r$ and $G'=\GEN{a^2}=\GEN{c^2}$.
	As in Case 1 we may assume that $w\ne 1$. Then both $o$ and $\frac{m}{r}$ are multiple of $4$ and we must prove, on the one hand that $y\equiv 1 \mod \frac{m}{2r}$ and, on the other hand that $y\equiv 1 \mod \frac{m}{r}$, if one of the following conditions hold: $o=n$ or, $s=m\ne nr$, or $2s=m=nr$. 
	From $4\mid o$ and $G/\GEN{c}$ being cyclic we deduce  $\GEN{c}=\GEN{b^{xv}a}$ with $o\mid v\mid n$ and $2\nmid x$. 
	From $G'=\GEN{a^2}=\GEN{c^2}$ it follows that $\frac{n}{2}\mid v$ so that $v$ is either $n$ or $\frac{n}{2}$. 
	If $v=n$ then $\GEN{c}=\GEN{a}$. 
	Therefore $a^{-1+r}=a^d=a^{(-1+r)^y}$ and hence $(-1+r)^{y-1} \equiv 1 \mod 2^m$. 
	Then $y\equiv 1 \mod \frac{m}{r}$ by \Cref{PropEse}.\eqref{opEspecial}.
	This proves the result if $o=n$ because in that case $v$ is necessarily $n$.
	
	Suppose otherwise that $v=\frac{n}{2}$. Then we distinguish the cases $m<nr$ and $m=nr$. 

%
	Assume that $m<nr$. Then, as $4\mid o\mid v$ we have $o_m(-1+r)=\max\left(2,\frac{m}{r}\right) \le \frac{n}{2}=v$ and hence $b^v$ is central in $G$. Then, having in mind that $4\mid r$ and $m\mid 2s$, we have 
	$$b^{xv}a^{(-1+r)^y}= (b^{xv}a)^{b^ya^z}=(b^{xv}a)^{-1+r}= b^{xv}a (b^{xv}a)^{r-2} 
	= b^{xv}a^{r-1+xs(\frac{r}{2}-1)} = b^{xv} a^{-1+s+r}.$$
	Therefore $(-1+r)^y \equiv -1+r+s \mod m$ and in particular $(-1+r)^y \equiv -1+r \mod s$, since $s\mid m$. 
	Using \Cref{PropEse} once more we deduce that $y\equiv 1 \mod \frac{m}{2r}$ and if $s=m$ then $y\equiv 1 \mod \frac{m}{r}$.
	
	Suppose otherwise that $m=nr$. Then, from \Cref{PropEse}.\eqref{vpEspecial} we have $v_2((-1+r)^v-1)=v_2(r)+v_2(v)=v_2(r)+v_2(n)-1=v_2(\frac{m}{2})$ so that $a^{b^v}=a^{1+\frac{m}{2}}$ and $(b^{xv}a)^2=a^{2+s+\frac{m}{2}}$ and hence $(b^{xv}a)^4=a^4$. 
	As $4\mid o$ it follows that $(b^{xv}a)^n=a^n$. 
	On the other hand, as $y$ is odd, it follows that $v_2((-1+r)^y+1)=v_2(r)\ge 2$, by \Cref{PropEse}.\eqref{vpEspecial+}. Therefore, $v_2(\Ese{(-1+r)^y}{n})=v_2(rn)-1=v_2(m)-1$, by \Cref{PropEse}.\eqref{vpEspecial}. 
	Then $\Ese{(-1+r)^y}{n}\equiv \frac{m}{2} \mod m$ an hence, having in mind that $8\mid \frac{m}{2}\mid s$ we deduce that $a^s=c^s=d^n=a^{ys+z\Ese{(-1+r)^y}{n}}=a^{s+z\frac{m}{2}}$. Therefore $z$ is even. 
	On the other hand from $c^d=c^{-1+r}$ and having in mind that $(-1+r)^v-1\equiv \frac{m}{2} \mod m$ and $z$ is even, we obtain
	$$b^{xv}a^{(-1+r)^y} =  (b^{xv}a)^{b^ya^z}=(b^{xv}a)^{-1+r}= b^{xv}a (b^{xv}a)^{r-2} 
	= b^{xv}a (a^{xs+2+\frac{m}{2}})^{\frac{r}{2}-1} = 
	b^{xv}a^{-1+s+r+\frac{m}{2}}.$$
	Therefore $(-1+r)^y \equiv -1+r+s+\frac{m}{2} \mod m$. 
	Again, from $m\mid 2s$ and \Cref{PropEse}.\eqref{opEspecial} we deduce that $y\equiv 1 \mod \frac{m}{2r}$ and if $s=\frac{m}{2}$ then $y\equiv 1 \mod \frac{m}{r}$.
	
	(2) implies (1). Suppose that $y\equiv 1 \mod w$. As $y$ is odd, if $o\mid 2$ then $b\in b^yN$ and hence $a$ and $b$ satisfy condition (1).
	So we assume from now on that $4\mid o$. In particular $4\mid n$. 
	Suppose that $m\mid 2r$, i.e. $r$ is either $m$ or $\frac{m}{2}$
	and let $c=a^y$ and $d=b^ya^2$. In this case $b^2$ is central in $P$ and hence $c^d=c^b=c^{-1+r}$ and applying statements \eqref{vpEspecial} and \eqref{vpEspecial+} of \Cref{PropEse} we obtain  $d^n=a^{ys+\Ese{(-1+r)^y}{n}}=a^{ys}=c^s$. Hence $c$ and $d$ satisfy the conditions of (1). 
	
	Thus from now on we assume that $4$ divides both $o$ and $\frac{m}{r}$. Suppose that $y\equiv 1 \mod \frac{m}{r}$. Then $a^{b^y}=a^b=a^{-1+r}$ because $b^{\frac{m}{r}}$ is central in $P$. 
	Moreover, as $m\mid 2s$ and $y$ is odd we have $(b^y)^n = a^{sy}=a^s$. Therefore $c=a$ and $d=b^y$ satisfy condition (1) and this finishes the proof of the lemma if $w=\frac{m}{r}$ and it also proves that for $w=\frac{m}{2r}$ we may assume that $y\not\equiv 1 \mod \frac{m}{r}$.
	So suppose that $w=\frac{m}{2r}$ and $y\not\equiv 1 \mod \frac{m}{r}$. 
	Then $y\equiv 1+\frac{m}{2r}\mod \frac{m}{r}$, $o<n$ and either $m=s=nr$ or $2s=m<nr$. Let $c=b^{\frac{n}{2}}a$ and $d=b^y$. 
	Then, in both cases, $c^2=a^{2+\frac{m}{2}}$ and, as $\frac{m}{2}$ is multiple of $4$ we have that $G'=\GEN{a^2}=\GEN{c^2}$, $|c|=m$ and $c^s=a^s$. 
	Moreover, 
	$$c^{-1+r}=(b^{\frac{n}{2}}a)^{-1+r} = 
	b^{\frac{n}{2}}a (b^{\frac{n}{2}}a)^{r-2} = 
	b^{\frac{n}{2}}a a^{(2+\frac{m}{2})(\frac{r}{2}-1)} = 
	b^{\frac{n}{2}}a^{-1+r+ \frac{m}{2}}= 
	b^{\frac{n}{2}} a^{(-1+r)(1+\frac{m}{2})} = 
	(b^{\frac{n}{2}}a)^{b^{1+\frac{m}{2r}}}=c^d$$
	and
	$$d^n = a^{s(1+\frac{m}{2r})} = a^s = c^s.$$	
	Then $c$ and $d$ satisfy the conditions of (1).
\end{proof}

\begin{theorem}\label{MetaTh}
	Let $m,n,s\in \N$ with $s\mid m$ and let $T$ and $\bar T$ be $(n,s)$-canonical cyclic subgroups of $\U_m$. 
	Set $[r,\epsilon,o]=[T]$, $[\bar r,\bar \epsilon,\bar o]=[\bar T]$, $\pi=\pi(r)\cup (\pi(n)\setminus \pi(m))$, $\bar \pi = \pi(\bar r)\cup (\pi(n)\setminus \pi(m))$, $m'=[T,n,s]$ and $\bar m'=[\bar T,n,s]$.
	
	Then the following statements are equivalent.
	\begin{enumerate}
		\item $\G_{m,n,s,T}$ and $\G_{m,n,s,\bar T}$ are isomorphic.
		\item $\Res_{m'}(T)=\Res_{\bar m'}(\bar T)$.
		\item $\pi=\bar \pi$, $\Res_{m_{\pi'}}(T_{\pi'})=\Res_{m_{\pi'}}(\bar T_{\pi'})$ and $\Res_{m_{\pi'}m'_p}(T_p)=\Res_{m_{\pi'}m'_p}(\bar T_p)$ for every $p\in \pi$.
	\end{enumerate}
\end{theorem}

\begin{proof}
	Let $G=\G_{m,n,s,T}$ and $\bar G=\G_{m,n,s,\bar T}$.
	To distinguish the generators $a$ and $b$ in the presentation of $G$ and $\bar G$ we denote the latter by $\bar a$ and $\bar b$. 
	We also denote $A=\GEN{a}$, $B=\GEN{b}$, $\bar A=\GEN{\bar a}$ and $\bar B=\GEN{\bar b}$. 
	The hypothesis warrants that $G=AB$ and $\bar G=\bar A \bar B$ are minimal metacyclic factorizations by \Cref{MiniMF}.
	In particular, $|A|=|\bar A|=m=m_G=m_{\bar G}$, $[G:A]=[\bar G:\bar A]=n=n_G=n_{\bar G}$, $[G:B]=[\bar G:\bar B]=s=s_G=s_{\bar G}$, $T=T_G(A)$ and $\bar T=T_{\bar G}(\bar A)$.
	
	(2) implies (3) Suppose that statement (2) holds. Then, using that $\pi(m)=\pi(m')=\pi(\bar m')$, we have $\Res_p(T)=\Res_p(\Res_{m'}(T))=\Res_p(\Res_{m'}(\bar T))=\Res_p(\bar T)$ for every prime $p$ dividing $m$. 
	Thus, $\pi'=\bar \pi'$ and, as $m_{\pi'}=m'_{\pi}$, we have $\Res_{m_{\pi'}}(T_{\pi'}) = \Res_{m'_{\pi'}}(T)_{\pi} = \Res_{m'_{\pi'}}(\bar T)_{\pi} = \Res_{m_{\pi'}}(\bar T_{\pi'})$ and 
	$\Res_{m_{\pi'}m'_p}(T_p) = \Res_{m'_{\pi'\cup \{p\}}}(T)_p = \Res_{m'_{\pi'\cup \{p\}}}(\bar T)_p = \Res_{m_{\pi'}m'_p}(\bar T_p)$ for every $p\in \pi(m)\setminus \pi'$.

	(1) implies (2). Suppose that $G\cong \bar G$. Then, as $T$ and $\bar T$ are $(n,s)$-canonical they yield the same parameters, i.e. $\pi'=\bar \pi'$, $o=\bar o$, etc.
	
	Let $f:\bar G\rightarrow G$ be an isomorphism and let $c=f(\bar a)$, $d=f(\bar b)$, $C=\GEN{c}$ and $D=\GEN{d}$. 
	Then $C_{\pi'}=f(\bar {G'}_{\pi'})={G'}_{\pi'}=A_{\pi'}$, by \Cref{Basic}.\eqref{BasicG'}.
	Furthermore, $C_{\pi'}D_{\pi'}=A_{\pi'}B_{\pi'}$ because $AB$ and $\bar A\bar B$ are the unique Hall $\pi'$-subgroup of $G$ and $\bar G$, respectively. 
	Then $\Res_{m_{\pi'}}(T)=T_G(A_{\pi'}) = T_G(C_{\pi'})=\Res_{m_{\pi'}}(\bar T)$. 
	As $\Res_{m_{\pi}}(T_{\pi'})=\Res_{m_{\pi}}(\bar T_{\pi'})=1$ it follows that $\Res_{m'}(T_{\pi'})=\Res_{m'}(\bar T_{\pi'})$. 
	Since $T$ and $\bar T$ are cyclic, it remains to prove that $\Res_{m'}(T_p)=\Res_{m'}(\bar T_p)$ for every $p\in \pi$. 
	Moreover, as $G$ and $\bar G$ have the same parameters $\epsilon$ and $r$ we have $\Res_{m_p}(T_p)=\Res_{m_p}(\bar T_p)=\GEN{\epsilon^{p-1}+r_p}_{m_p}$. 
	Denote $R=\epsilon^{p-1}+r_p$ and select  generators $t$ of $\Res_{m_{\pi'}m'_p}(T_p)$ and $\bar t$ of $\Res_{m_{\pi'}m'_p}(T_p)$ such that $\Res_{m_p}(t)=\Res_{m_p}(\bar t) [R]_{m_p}$.  
	We already know that $\Res_{m'_{\pi'}}(T)=\Res_{m'_{\pi'}}(\bar T)$ and in particular, there is an integer $x$ coprime with $p$ such that $\bar t=t^x \mod m_{\pi'}$. 
	If $o_p\le 2$ then $\Res_{m'_{\pi'}}(t)=\Res_{m'_{\pi'}}(\bar t)$ and if $o_{m'_p}(R) \le 2$ then $\Res_{m'_p}(t^x)=[R^x]_{m'_p}=[R]_{m_p}=\Res_{m'_p}(\bar t)$. 
	In both cases $\Res_{m_{\pi'}m'_p}(T)=\GEN{t}=\GEN{t^x}=\Res_{m_{\pi'}m'_p}(\bar T)$, as desired. 
	Therefore, in the remainder we may assume that both $o_p$ and $o_{m'_p}(R)$ are greater than $2$ and, in particular, 
	$o_{m'_p}(R)=\frac{m'_p}{r_p}=\Res_{m'_p}(T)$ and this number coincides with the $w$ in \Cref{cdw}.
	
	On the other hand $A_pB_p$ and $f(\bar A_p \bar B_p)=C_pD_p$ are Sylow $p$-subgroup of $G$ and hence they are conjugate in $G$. 
	Composing $f$ with an inner automorphism of $G$ we may assume that $C_pD_p=A_pB_p$.
	Then $\GEN{c,d^{o_p}} =  f(\GEN{\bar a,\bar b^{o_p}}) = 
	f(C_{\bar G_p}(\bar {G'}_{\pi'}))=C_{G_p}({G'}_{\pi'})=\GEN{a,b^{o_p}}$.
	By \Cref{cdw} we have $d=b^yg$ for some $g\in C_{G_p}({G'}_{\pi'})$ and $y\equiv 1 \mod w$.
	Thus $\Res_{m_{\pi'}}(\bar t)=\Res_{m_{\pi'}}(t^y)$ and $\Res_{m'_p}(\bar t) = \Res_{m'_p}(t) = \Res_{m'_p}(R) = \Res_{m'_p}(R^y) = \Res_{m'_p}(t^y)$, because $y\equiv 1 \mod o_{m'_p}(R)$. Thus $\Res_{m'_{\pi'}m'_p}(\bar T_p)=\Res_{m'_{\pi'}m'_p}(\bar t)= \Res_{m'_{\pi'}m'_p}(t^y)= \Res_{m'_{\pi'}m'_p}(T_p)$, as desired. 
	
	(3) implies (1) Suppose that the conditions of (3) holds. We may assume that $a=\bar a$ and take generators $t$ of $T$ and $\bar t$ of $\bar T$ so that $G=\GEN{a,b}$, $\bar G=\GEN{a,\bar b}$, with $|a|=m$, $[G:\GEN{a}]=n$, $b^n=a^s$, $a^b=a^t$, $a^{\bar b}=a^{\bar t}$.
	Moreover, from the assumption we may assume $a^{b_{\pi'}}=a^{\bar b_{\pi'}}$ and for every $p\in \pi$ we have 
	$\Res_{m_{\pi'}m'_p}(T_p)=\Res_{m_{\pi'}m'_p}(\bar T_p)$.
	In particular, for every $p\in \pi$, we have $\GEN{\epsilon^{p-1}+r_p}_{m'_p}=\Res_{m'_p}(T_p)=\Res_{m'_p}(\bar T_p)=\GEN{\bar \epsilon^{p-1}+\bar r_p}$. 
	Since $r_p\mid m'_p\mid m_p$ it follows that $\epsilon=\bar \epsilon$ and $r_p=\bar r_p$. 
	Thus $r=\bar r$. 
	
	We claim that for every $p\in\pi$ we can rewrite $G_p=\GEN{a_p,b_p}$ as $G_p=\GEN{c_p,d_p}$ with $c_p\in \GEN{a_p,b_p^{o_p}}=C_{G_p}(a_{\pi'})$ and $d_p\in b^yC_{G_p}(a_{\pi'})$ such that $|c_p|=m_p$, $c_p^{d_p}=c_p^{R_p}$, $a_{\pi'}^{d_p}=a_{\pi'}^{\bar b_p}$ and $d_p^{n_p}=c_p^{s_p}$.
	
	Indeed, let $p\in \pi$. The assumption $\GEN{\Res_{m_{\pi'}m'_p}(t_p)}=\GEN{\Res_{m_{\pi'}m'_p}(\bar t_p)}$ implies that there is an integer $y$ coprime with $|\Res_{m_{\pi'}m'_p}(t_p)|$ such that $\Res_{m_{\pi'}m'_p}(\bar t_p)=\Res_{m_{\pi'}m'_p}(t_p)^y$. 
	If $o_p\le 2$ or $o_{m_p}(R)\le 2$ then, as in the proof of (1) implies (2) we have that $\Res_{m_{\pi'}m_p}(t)=\Res_{m_{\pi'}m_p}(\bar t)$ so that $c_p=a_p$ and $d_p=b_p$ satisfies the desired conditions. So assume that $o_p>2$ and $o_{m_p}(R)>2$. 
	From the equality $a_p^{b_p}=a_p^{\bar b_p}$ we deduce that $R^y \equiv R \mod m'_p$ 
	and this implies that $y\equiv 1 \mod w$ where $w=o_{m'_p}(R)=\frac{m'_p}{r_p}$ and again this $w$ coincides with the one in \Cref{cdw}.
	Applying \Cref{cdw} we deduce that $\GEN{a_p,b_p}$ contain elements $c_p\in \GEN{a_p,b_p^o}=C_{G_p}(a_{\pi'})$ and $d_p\in b^yC_{G_p}(a_{\pi'})$ such that $\GEN{a_p,b_p}=\GEN{c_p,d_p}$, $|c_p|=m_p$,  $a_{\pi'}^{d_p}=a_{\pi'}^{b_p^y}=a_{\pi'}^{\bar b_p}$, $c_p^{d_p}=c_p^{R_p}$ and $d_p^{n_p}=c_p^{s_p}$, as desired. This finishes the proof of the claim. 
	
	For every $p\in \pi$ let $c_p$ and $d_p$ as in the claim and set $c = a_{\pi'}\prod_{p\in \pi} c_p$ and $d=b_{\pi'}\prod_{p\in \pi} d_p$ we deduce that $G=\GEN{c,d}$ with $|c|=m$, $d^n=c^s$ and $c^d=a^{\bar t}$. Therefore $G\cong \bar G$. 
\end{proof}

The following corollary is a direct consequence (1) implies (2) of \Cref{MetaTh}. It shows that $\Delta_G$ is well defined.

\begin{corollary}\label{DeltaWD}
If $G=AB=CD$ are two minimal factorizations of $G$ then $\Delta(AB)=\Delta(CD)$.
\end{corollary}

\section{Proofs of the main results}\label{SectionProofs}

\begin{proofof}\emph{\Cref{Invariants}}.
Let $G$ and $\bar G$ be finite metacyclic groups and let $G=AB$ and $\bar G=\bar A \bar B$ be minimal metacyclic factorizations of $G$ and $\bar G$ respectively. Denote $m=|A|$, $\bar m=|\bar A|$, $n=[G:A]$, $\bar n=[\bar G:\bar A]$, $s=[G:B]$, $\bar s = [\bar G:\bar B]$, $T=T_G(A)$ and $\bar T = T_{\bar G}(\bar A)$. We also denote $m'=[T,n,s]$, $\bar m'=[\bar T, \bar n,\bar s]$, $\Delta=\Res_{m'}(T)$ and $\bar \Delta=\Res_{\bar m'}(\bar T)$.
Then $G\cong \G_{m,n,s,T}$, $\bar G \cong \G_{\bar m,\bar n, \bar s, \bar T}$, $m=m_G$, $n=n_G$, $s=s_G$, $\bar n=n_{\bar G}$, $\bar m=m_{\bar G}$, $s=s_{\bar G}$, $T$ is $(n,s)$-canonical and $\bar T$ is $(\bar n,\bar s)$-canonical. 
Moreover, $\Delta=\Delta_G$ and $\bar \Delta = \Delta_{\bar G}$. 

If $G\cong G'$ then $m=\bar m$, $n=\bar n$, $s=\bar s$ and, by \Cref{MetaTh} we have $\Delta=\bar \Delta$. Thus $\INV(G)=\INV(\bar G)$.

Conversely, if $\INV(G)=\INV(\bar G)$ then $m=|A|=m_G=m_{\bar G}=|\bar A|=\bar m$ and similarly $n=\bar n$ and $s=\bar s$. Moreover, $\Res_{m'}[T]=\Delta_G=\Delta_{\bar G}=\Res_{\bar m'}(\bar T)$. Then $G\cong \bar G$ by \Cref{MetaTh}.
\end{proofof}

In the remainder of the section we use the notation in \Cref{Parameters}.\medskip

\begin{proofof}\emph{(1) implies (2) in \Cref{Parameters}}.
Suppose that $(m,n,s,\Delta)=\INV(G)$ for some metacyclic group $G$ and let $G=AB$ be a minimal factorization of $G$. Then $m=m_G=|A|$, $n=n_G=[G:A]$, $s=s_G=[G:B]$ and if $T=T_G(A)$ then $\Delta=\Delta(AB)=\Res_{m'}(T)$. In particular, $s\mid m$,  $T$ is a cyclic subgroup of $\U_m^{n,s}$, $\inv{T}=\inv{\Delta}$ and $m'_{\nu}=m_{\nu}$. Moreover, $\nu=\pi(m')\setminus \pi(r)$ and $s_{\nu}=m_{ \nu}$, by \Cref{Basic}. Moreover, $|\Delta |$ divides $n$, because it divides $|T|$, which in turn divides $n$.
Then conditions \eqref{ParamB} and \eqref{Param'} of \Cref{Parameters} hold.
By \Cref{ParametersTDelta}, \Cref{Basic} and \Cref{epsilonWD} we have $\pi=\pi_G$,
$\pi'_G=\nu$, $o=o_G$, $\epsilon=\epsilon_G$ and $r=r_G$.
Let $p\in \pi(r)$.
If $\epsilon^{p-1}=1$ then $\frac{m_p}{r_p}=|\Res_{m_p}(T_p)|\le n_p$ and if $\epsilon=-1$ then $\max(2,\frac{m_2}{r_2})=|\Res_{m_2}(T_2)|\le |T_2|\le n_2$ and $m_2\le 2s_2$.
As the metacyclic factorization $G=AB$ is minimal, $T$ is $(n,s)$-canonical by \Cref{MiniMF}. Then the remaining conditions in \eqref{Param-} and \eqref{Param+} follow.
\end{proofof}

\begin{proofsof}\emph{\Cref{Construction} and (2) implies (1) in \Cref{Parameters}}.
Suppose that $m,n,s$ and $\Delta$ satisfy the conditions of (2) in \Cref{Parameters}.
By \Cref{ConsRemark} there is a cyclic subgroup $T$ of $\U_m^{n,s}$ with $\Res_{m'}(T)=\Delta$ and $[T]=[\Delta]$. 
Let $t\in \N$ with $T=\GEN{t}_{m}$.
Let $G=\G_{m,n,s,t}$ and denote $A=\GEN{a}$ and $B=\GEN{b}$.  We will prove that $G=AB$  is a minimal factorization of $G$ that $m=|A|$, $n=[G:A]$, $s=[G:B]$ and $\Delta=\Delta(AB)$.
This will complete the proofs of \Cref{Parameters} and \Cref{Construction}.

Of course $G=AB$ is a metacyclic factorization of $G$ and $T=T_G(A)$.
Since $m_{\nu}=s_{\nu}$, $n$ is multiple of $|\Delta|$ and  $|\Res_{m_{\nu}}(T)|=|\Res_{m_{\nu}}(\Delta)|$, it follows that $|\Res_{m_{\nu}}(T)|$ divides $n$ and $s(t-1)$. 
On the other hand if $p\mid r$ then $t\equiv \epsilon^{p-1}+r_p \mod m_p$. Therefore, if $\epsilon^{p-1}=1$ then $o_{m_p}(t)=\frac{m_p}{r_p}\mid n$ and $s(t-1)\equiv sr_p\equiv 0 \mod m_p$. 
Otherwise, i.e. if $\epsilon=-1$ and $p=2$, then $2\mid |\Delta|\mid n$ and $\frac{m_2}{r_2}\le n_2$ and $m_2\mid 2s$. Thus $o_{m_2}(t)=o_{m_2}(-1+r_2)=\max(2,\frac{m_2}{r_2})\le n_2$ and $m_2\mid t(s-1)$. This shows that $m$ divides both $t^n-1$ and $s(t-1)$, i.e. $T\subseteq \U_m^{n,s}$.
Then $|A|=m$ and $[G:A]=n$, and hence $[G:B]=s$. From condition \eqref{Param'} we have that $\Delta=\Res_{m'}(T_G(A))=\Delta(AB)$ and  from conditions \eqref{Param+} and \eqref{Param-} it follows that $T$ is $(n,s)$-canonical. Then the metacyclic factorization $G=AB$ is minimal by \Cref{MiniMF}.
\end{proofsof}

%
%


Having in mind that a metacyclic group is nilpotent if and only if $o_G=1$ one can easily obtain from \Cref{Parameters}  a description of the finite nilpotent metacyclic groups or equivalently the values of the lists of metacyclic invariants of the finite nilpotent metacyclic groups. Observe that (1) corresponds to cyclic groups, (2) to 2-generated abelian groups, (3) to non-abelian nilpotent metacyclic groups $G$ with $\epsilon_G=1$ and (4) to metacyclic nilpotent groups with $\epsilon_G=-1$.

\begin{corollary}\label{Nilpotent}
	Let $m,n,s\in \N$ and $t\in \N\cup \{0\}$. Then $(m,n,s,t)$ is the list of metacyclic invariants of a finite metacyclic nilpotent group if and only if $s\mid m$, $t<m$ and one of the following conditions hold:
	\begin{enumerate}
		\item $m=1$.
		\item $t=1$ and $s=m\le n$. 
		\item $\pi(t-1)=\pi(m)$, $\lcm\left(t-1,\frac{m}{t-1}\right)\mid s \mid n$ and if $4\mid m$ then $4\mid t-1$.
		\item There is a divisor $r$ of $s_{2'}m_2$ such that $\pi(r)=\pi(m)$, $4\mid r$, 
		$t \equiv 1+r_{2'}\mod m_{2'}$,  $t\equiv -1+r_2 \mod m_2$, $\frac{m_{2'}}{r_{2'}} \mid s_{2'}\mid n_{2'}$,  $\max\left(2,\frac{m_2}{r_2}\right)\le n_2$, $m_2\le 2s_2$ and $s_2\ne n_2r_2$. If moreover $4\mid n$ and $8\mid m$ then $r_2\le s_2$.
	\end{enumerate}
	In that case $\G_{m,n,s,t}$ is nilpotent with metacyclic invariants $(m,n,s,t)$.
\end{corollary}

\section{A GAP implementation}\label{SectionImplementation}

In this section we show how we can use the result in previous sections to construct some GAP functions for calculations with finite metacyclic groups.
The code of these function is available in \cite{Metacyclics}.

We start with two auxiliar functions.
We call \emph{metacyclic parameters} to any list $(m,n,s,t)$ with $m,n,s\in \N$ and $[t]_m\in \U_m^{n,s}$, i.e. $s(t-1)\equiv t^n-1 \mod m$. In that case \texttt{MetacyclicGroupPC([m,n,s,t])} outputs the group $\G_{m,n,s,t}$ with a power-conjugation presentation.
The boolean function \texttt{IsMetacyclic} returns \texttt{true} if the input is a finite metacyclic and \texttt{false} otherwise.
\begin{verbatim}
gap> G:=MetacyclicGroupPC([10,20,5,3]);
<pc group of size 200 with 5 generators>
gap> IsMetacyclic(G);
true
gap> Filtered([1..16],x->IsMetacyclic(SmallGroup(100,x)));
[ 1, 2, 3, 4, 5, 6, 8, 9, 14, 16 ]
\end{verbatim}

To introduce the next function we start presenting an algorithm that uses \Cref{AlgoMinMF} to compute $\INV(G)$ for a given metacyclic group $G$. Observe that in \Cref{AlgoMinMF} the values of $m=|a|$, $n=[G:\GEN{a}]$, $s=[G:\GEN{a}]$ and $(r,\epsilon,o)=[T_G(\GEN{a})]$ are updated along the calculations. 
We use this in step \eqref{UseAlgoMinMF} of the following algorithm. 

\begin{algorithm}\label{AlgoINV}
	{\sc Input}: A finite metacyclic group $G$.
	
	{\sc Output}: $\INV(G)$.
	
	\begin{enumerate}
		\item Compute a metacyclic factorization $G=AB$ of $G$.
		\item\label{UseAlgoMinMF} Perform \Cref{AlgoMinMF} with input $(A,B)$ saving not only the output $(\GEN{a},\GEN{b})$ but also $m,n,s,r,\epsilon$ and $o$ computed along.
		\item Compute $m'$ using \eqref{m'} and $t\in \N$ such that $a^b=a^t$. 
		\item Return $(m,n,s,\Res_{m'}(\GEN{t}_m))$. 
	\end{enumerate}
\end{algorithm}

A slight modification of \Cref{AlgoINV} allows the computation of the list of metacyclic invariants of a finite metacyclic group: 

\begin{algorithm}\label{AlgoMC}
	{\sc Input}: A finite metacyclic group $G$.
	
	{\sc Output}: The list of metacyclic invariants of $G$.
	
	\begin{enumerate}
		\item Compute a metacyclic factorization $G=AB$ of $G$.
		\item\label{Algo2MF} Perform \Cref{AlgoMinMF} with input $(A,B)$ saving not only the output $(\GEN{a},\GEN{b})$ but also $m,n,s,r$ and $\epsilon$ computed along.
		\item\label{Algo2t} Compute $m'$ using \eqref{m'} and $t\in \N$ such that $a^b=a^t$ and set $\Delta:=\Res_{m'}(\GEN{t}_m)$. 
		\item\label{Algo2tp} Use the Chinese Remainder Theorem to compute the unique $1\le t \le m_{\pi(r)}$ such that $t\equiv \epsilon^{p-1}+r_p \mod m_p$ for every $p\in \pi(r)$. 
		\item While $\gcd(t,m')\ne 1$ or $\GEN{t}_{m'}\ne \Delta$, $t:=t+m_{\pi(r)}$. 
		\item Return $(m,n,s,t)$.
	\end{enumerate}
\end{algorithm}

%
%

Observe that $G=\GEN{a}\GEN{b}$ is a minimal metacyclic factorization at step \eqref{Algo2MF} of \Cref{AlgoMC},   and $m=m_G$, $n=n_G$ and $s=s_G$.
At step \eqref{Algo2t}, we have $T_G(\GEN{a})=\GEN{t}_m$ and hence $G\cong \G_{m,n,s,t}$ and $\Delta=\Delta_G=\Res_{m'}(\GEN{t}_m)$. However, this $t$ is not $t_G$ yet. 
The $t$ at step \Cref{Algo2tp} is the smallest one with $t\equiv \epsilon^{p-1}+r_p \mod m_p$ for every $p\in \pi(r)$ and the next steps search for the first integer $t$ satisfying this condition as well as representing an element of $\U_m$ with $\Res_{m'}(\GEN{t}_m)=\Delta$. 

The GAP function \texttt{MetacyclicInvariants} implements \Cref{AlgoMC}.
For example in the following calculations one computes the metacyclic invariants of all the metacyclic groups of order $200$.

\begin{verbatim}
gap> mc200:=Filtered([1..52],i->IsMetacyclic(SmallGroup(200,i)));;
gap> List(mc200,i->MetacyclicInvariants(SmallGroup(200,i)));
[[25,8,25,24],[1,200,1,0],[25,8,25,7],[100,2,50,99],[100,2,50,49],[100,2,100,99],
[50,4,50,49],[2,100,2,1],[4,50,4,3],[4,50,2,3],[50,4,50,7],[5,40,5,4],[5,40,5,1],
[5,40,5,2],[20,10,10,19],[20,10,10,9],[20,10,20,19],[10,20,10,9],[10,20,10,1],
[20,10,20,11],[20,10,10,11],[10,20,10,3]]
\end{verbatim}

The GAP functions \texttt{MCINV} and \texttt{MCINVData} implement \Cref{AlgoINV} representing $\INV(G)$ in two different ways. 
While \texttt{MCINV(G)} outputs $\INV(G)$ if $G$ is a metacyclic group, \texttt{MCINVData(G)} ouputs a 5-tuple \texttt{[m,n,s,m',t]} such that $\INV(G)=(m,n,s,\GEN{t}_{m'})$.
The input data \texttt{G} can be replaced by metacyclic parameters $[m,n,s,t]$ representing the group $\G_{m,n,s,t}$:

\begin{verbatim}
gap> G:=SmallGroup(384,533);           
<pc group of size 384 with 8 generators>
gap> MetacyclicInvariants(G);
[ 8, 48, 4, 5 ]
gap> x:=MCINV(G);                     
[ 8, 48, 4, <group of size 1 with 1 generator> ]
gap> y:=MCINVData(G);
[ 8, 48, 4, 4, 1 ]
gap> x[4]=Group(ZmodnZObj(y[5],y[4]));
true
gap> H:=MetacyclicGroupPC([8,48,4,5]);
<pc group of size 384 with 8 generators>
gap> IdSmallGroup(H);
[ 384, 533 ]
gap> MetacyclicInvariants([20,4,8,11]);
[ 4, 20, 4, 3 ]
gap> MCINVData([20,4,8,11]);
[ 4, 20, 4, 4, 3 ]
\end{verbatim}

Observe that two finite metacyclic groups $G$ and $H$ are isomorphic if and only if $\INV(G)=\INV(G)$ if and only if they have the same metacyclic invariants. 
The function \texttt{AreIsomorphicMetacyclicGroups} uses this to decide if two metacyclic groups $G$ and $H$ are isomorphic. 
It outputs \texttt{true} if $G$ and $H$ are isomorphic finite metacyclic groups and \texttt{false} if they are finite metacyclic groups but they are not isomorphic.
In case one of the inputs is not a finite metacyclic group then it fails.
The input data \texttt{G} and  \texttt{H} can be replaced by metacyclic parameters of them.
\begin{verbatim}
gap> H:=MetacyclicGroupPC([100,30,10,31]);
<pc group of size 3000 with 7 generators>
gap> K:=MetacyclicGroupPC([300,30,10,181]);
<pc group of size 9000 with 8 generators>
gap> AreIsomorphicMetacyclicGroups(H,K);
false
gap> AreIsomorphicMetacyclicGroups([300,10,10,31],K);
false
gap> G:=MetacyclicGroupPC([300,10,10,31]);
<pc group of size 3000 with 7 generators>
gap> MetacyclicInvariants(G);
[ 100, 30, 10, 31 ]
gap> MetacyclicInvariants(H);
[ 100, 30, 10, 31 ]
gap> MetacyclicInvariants(K);
[ 50, 180, 10, 31 ]
\end{verbatim}

We now explain a method to compute all the metacyclic group of a given order $N$.
We start producing all the tuples $(m,n,s,r,\epsilon,o)$ such that $\INV(G)=(m,n,s,\Delta)$ and $\inv{\Delta}=(r,\epsilon,o)$ for some finite metacyclic group $G$ and some cyclic subgroup $\Delta$ of $\U_{m'}$ with $m'$ as in \eqref{m'}. For such group $G$ we denote $\IN(G)=(m,n,s,r,\epsilon,o)$. The following lemma characterizes when a given tuple $(m,n,r,s,r,\epsilon,o)$ equals $\IN(G)$ for some finite metacyclic group:

\begin{lemma}\label{mnsrep}
Let $m,n,s,r,o\in \N$ and $\epsilon\in \{1,-1\}$ and let $\pi'=\pi(m)\setminus \pi(r)$ and $\pi=\pi(mn)\setminus \pi'$.
Then $\IN(G)=(m,n,s,r,\epsilon,o)$ for some finite metacyclic group $G$ if and only if the following conditions hold:
\begin{enumerate}[label=(\Alph*)]
			\item\label{CondDiv} $s\mid m$, $r\mid m$, $o\mid n_{\pi}$, $m_{\pi}\mid rn$, $m_{\pi} \mid rs$, $s_{\pi'}=m_{\pi'}$ and if $4\mid m$ then $4\mid r$.
			\item\label{Cond+} If $p\in \pi(r)$  and $\epsilon^{p-1}=1$ then $s_p\mid n$ and either $r_p\mid s$ or $s_po_p\nmid n$.
			\item\label{Cond-} If $\epsilon=-1$ then $2\mid n$, $4\mid m$, $m_2\mid 2s$, $s_2\ne n_2r_2$. If moreover $4\mid n$, $8\mid m$ and $o_2<n_2$ then $r_2\mid s$.
			\item\label{opi'} $o\mid \lcm \{ q-1 : q\in \pi'\}$ and for every $q\in \pi'$ with $\gcd(o,q-1)=1$ there is $p\in \pi'\cap \pi(n)$ with $p\mid q-1$.
		\end{enumerate}
\end{lemma}

\begin{proof}
Suppose first that $(m,n,s,r,\epsilon,o)=\IN(G)$ for some finite metacyclic group $G$. Then $\INV(G)=(m,n,s,\Delta)$ for some cyclic subgroup $\Delta$ of $\U_{m'}$ with $\inv{\Delta}=(r,\epsilon,o)$. Then the conditions in statement (2) of \Cref{Parameters} hold and this implies that conditions \ref{CondDiv}--\ref{Cond-} hold.
To prove \ref{opi'} we fix a metacyclic factorization $G=AB$ and observe that $o=o_G(A)=|\Res_{m_{\pi'}}(T_G(A))_{\pi}|$ and $\Res_{m_{\pi'}}(T_G(A))_{\pi}$ is a cyclic subgroup of $(\U_{m_{\pi'}})_{\pi}$.
Then $o$ divides the exponent of $(\U_{m_{\pi'}})_{\pi}$ which is $\lcm\{(q-1)_{\pi} : q\in \pi'\}$.
This proves the first part of \ref{opi'}.
To prove the second one we take $q\in \pi'$ such that $\gcd(o,q-1)=1$.
By \Cref{Basic}.\eqref{BasicResume}, we have $\Res_{q}(T_G(A))\ne 1$.
However $\Res_{q}(T_G(A))_{\pi}\mid \gcd(o,q-1)=1$ and hence, if $p$ is a divisor of $\Res_{q}(T_G(A))$ then $p\mid |U_q|=q-1$, $p\mid [G:A]=n$ and $p\not\in \pi$, so that $p\in \pi'$. This finishes the proof of \ref{opi'}.

Conversely, suppose that conditions \ref{CondDiv}-\ref{opi'} hold.
By condition \ref{opi'}, $2\not\in \pi'$ and hence if $q\in \pi'$ then $\U_{m_q}$ is cyclic of order $\varphi(m_q)$. Therefore for every $q\in \pi'$, the group $\U_q$ contains a cyclic subgroup of order $q-1$.
Therefore $\U_m$ contains a cyclic subgroup of order $k=\lcm\{q-1 : q\in \pi'\}$.
Furthermore, by \ref{opi'}, for every $p\in \pi$ we have that $o_p\mid k$ and hence $o_p\mid q-1$ for some $q\in \pi'$.
Then $\U_{m_q}$ contains an element of order $o_p$ and, as $\U_{m_{\pi'}}\cong \prod_{q\in \pi'} \U_{m_q}$, it follows that $\U_{m_{\pi'}}$ contains an element of order $o$.
Let $\tau=\{q\in \pi' : \gcd(o,q-1)=1\}$. By \ref{opi'}, for every $q\in \tau$ there is $p_q\in \pi'\cap \pi(n)$ such that $p_q\mid q-1$.
Let $h=\prod_{q\in \tau} p_q$.
For every $q\in \tau$, there is an element in $\U_{m_q}$ of order $p_q$.
Then $\U_{m_\tau}$ has an element of order $h$.
As $o\mid n_\pi$ and $h\mid n_{\pi'}$, $\U_{m_{\pi'}}$ has a cyclic subgroup $S$ of order $oh$.
Then $\Aut(C_m)$ has a cyclic subgroup $T$ such that $\Res_{m_{\pi'}}(T)=S$ and
$\Res_{m_p}(T)=\Res_{m_p}(T)=\GEN{\epsilon^{p-1}+r_p}_{m_p}$ for every $p\in \pi$.
By condition \ref{Cond+}, if $p\in \pi(r)$ and $\epsilon^{p-1}=1$ then $|\Res_{m_p}(T)|=\frac{m_p}{r_p}\mid n_p$.
By condition \ref{Cond-}, if $\epsilon=-1$ then $2\in \pi$, $2\mid n$ and $\frac{m_2}{r_2}\mid n$ by \ref{CondDiv}.
Thus $|\Res_{m_p}(T)|=\max(2,\frac{m_2}{r_2})\mid n$.
Then $|\Res_{m_p}(T)|$ divides $n$ for every $p\in \pi$.
This implies that $|T|=\lcm(|S|,|\Res_{m_p}(T)|,p\in \pi)$ and this number divides $n$.
On the one hand we have $s_{p'}=m_{\pi'}$ and if $p\in \pi$ then either $m_p\mid rs$ or $p=2$, $\epsilon=-1$ and $2m_2\mid s$. Using this it is easy to see that $\Res_{\frac{m}{s}}(T)=1$. This proves that $T\subseteq \U_m^{n,s}$ and by the election of $T$ it follows that $\inv{T}=(r,\epsilon,o)$. Moreover, from conditions \ref{Cond+} and \ref{Cond-}, it follows that $T$ is $(n,s)$-canonical and hence $\G_{m,n,s,T}=\GEN{a}\GEN{b}$ is a minimal factorization. Thus $\IN(\G_{m,n,s,T})=(m,n,s,r,\epsilon,o)$, as desired.
\end{proof}

Our last algorithm is based in \Cref{mnsrep} and compute a list containing exactly one representative of each isomorphism class of the metacyclic groups of a given order.

\begin{algorithm}\label{AlgoAll}
	{\sc Input}: A positive integer $N$.

	{\sc Output}: A list containing exactly one representative of each isomorphism class of the  metacyclic groups of order $N$.
	\begin{enumerate}
		\item $M:=[\;]$, an empty list, $\pi':=\pi(m)\setminus \pi(r)$, $\pi':=\pi(N)\setminus \pi'$.
		\item $P:=\{(m,n,s,r,\epsilon,o) : n,m,s,r,o\in \N, \epsilon\in \{1,-1\}, N=mn \text{ and conditions \ref{CondDiv}-\ref{opi'} hold} \}$.
		\item For each $(m,n,s,r,\epsilon,o) \in P$:
		\begin{enumerate}
			\item\label{m's'} $m':=m_{\pi'}\prod_{p\in \pi(r)} m'_p$ with $m'_p$ as in \eqref{m'} and $s':=\frac{sm'}{m}$.
			\item\label{T} For every cyclic subgroup $\Delta$ of $\U_{m'}^{n,s'}$ with $\inv{\Delta}=(r,\epsilon,o)$:
			\begin{itemize}
				\item Select a cyclic subgroup $T$ of $\U_m$ such that $\Res_{m'}(T)=\Delta$.
				\item Add $\G_{m,n,s,T}$ to the list $M$.
			\end{itemize}
		\end{enumerate}
		\item Return the list $M$.
	\end{enumerate}
\end{algorithm}

Observe that if $(m,n,s,r,\epsilon,o)$ satisfy conditions \ref{CondDiv}-\ref{opi'} then $m$ divides $sm'$.
Indeed, if $p\nmid r$ then $m_p=m'_p$. If $\epsilon=-1$ then $\frac{m_2}{2}$ divides $s$ and $2\mid m'$, hence in this case $\frac{m_2}{s_2}\mid m'$.
Finally, if $p\in \pi(r)$ and $\epsilon^{p-1}=1$. Then $p\in \pi$ and hence $m_p\le r_ps_p$ by condition \ref{CondDiv}.
Therefore $\frac{m_p}{s_p}\le \min(m_p,r_po_p)$.
If $r_p\mid s_p$ then also $\frac{m_p}{s_p}\le s_p$.
Otherwise $s_po_p\nmid n$ and hence $r_p\frac{s_po_p}{n_p}>r_p\ge \frac{m_p}{s_p}$.
This proves that $\frac{m_p}{s_p}\mid m'$ for every prime $p$, so that $m\mid sm'$, as desired.
This justify that $s'\in \N$ is step \eqref{m's'}.

On the other hand if $T$ is as in \eqref{T} then $T\subseteq \U_m^{n,s}$. Indeed, $\frac{m}{s}=\frac{m'}{s'}$ and hence $\Res_{\frac{m}{s}}(T)=\Res_{\frac{m'}{s'}}(\Delta)=1$. Moreover $\Res_{m_{\pi'}}(T)=\Res_{m'_{\pi'}}(\Delta)$ and hence
$|\Res_{m_{\pi'}}(T)|$ divides $n$. On the other hand $\inv{T}=(r,\epsilon,o)=\inv{T}$ and hence if $\epsilon^{p-1}=1$ then $|\Res_{m_p}(T)|=\frac{m_p}{r_p} \mid n$, by \ref{CondDiv}. Otherwise $|\Res_{m_2}{T_2}|=\max(2,\frac{m_2}{r_2})$ which divides $n$ by \ref{CondDiv} and \ref{Cond-}.

The function \texttt{MetacyclicGroupsByOrder(N)} implements a combination of \Cref{AlgoMC} and \Cref{AlgoAll} and returns the complete list of metacyclic invariants of metacyclic groups of order $N$.

\begin{verbatim}
gap> MetacyclicGroupsByOrder(200);
[[1,200,1,0],[2,100,2,1],[4,50,2,3],[4,50,4,3],[5,40,5,1],[5,40,5,2],[5,40,5,4],
[10,20,10,1],[10,20,10,3],[10,20,10,9],[20,10,10,9],[20,10,10,11],[20,10,10,19],
[20,10,20,11],[20,10,20,19],[25,8,25,7],[25,8,25,24],[50,4,50,7],[50,4,50,49],
[100,2,50,49],[100,2,50,99],[100,2,100,99]]
gap> MetacyclicGroupsByOrder(8*3*5*7);
[[1,840,1,0],[2,420,2,1],[3,280,3,2],[4,210,2,3],[4,210,4,3],[5,168,5,2],[5,168,5,4],
[6,140,6,5],[7,120,7,2],[7,120,7,6],[7,120,7,3],[10,84,10,3],[10,84,10,9],[12,70,6,5],
[12,70,6,11],[12,70,12,11],[14,60,14,3],[14,60,14,9],[14,60,14,13],[15,56,15,2],
[15,56,15,14],[20,42,10,9],[20,42,10,19],[20,42,20,19],[21,40,21,20],[28,30,14,3],
[28,30,14,5],[28,30,14,11],[28,30,14,13],[28,30,14,27],[28,30,28,3],[28,30,28,11],
[28,30,28,27],[30,28,30,17],[30,28,30,29],[35,24,35,2],[35,24,35,3],[35,24,35,4],
[35,24,35,13],[35,24,35,19],[35,24,35,34],[42,20,42,41],[60,14,30,29],[60,14,30,59],
[60,14,60,59],[70,12,70,3],[70,12,70,9],[70,12,70,13],[70,12,70,19],[70,12,70,23],
[70,12,70,69],[84,10,42,41],[84,10,42,83],[84,10,84,83],[105,8,105,62],[105,8,105,104],
[140,6,70,9],[140,6,70,19],[140,6,70,39],[140,6,70,69],[140,6,70,89],[140,6,70,139],
[140,6,140,19],[140,6,140,39],[140,6,140,139],[210,4,210,83],[210,4,210,209],
[420,2,210,209],[420,2,210,419],[420,2,420,419]]
\end{verbatim}

	\bibliographystyle{amsalpha}
	\bibliography{ReferencesMSC}	
\end{document}